\documentclass[11pt,a4paper,review]{siamart171218}
\usepackage[utf8]{inputenc}
\usepackage[english]{babel}

\usepackage{graphicx}
\usepackage{fancyhdr}

\usepackage{amsfonts}
\usepackage{amssymb}

\usepackage{amsmath}
\usepackage{amsfonts}
\usepackage{amssymb}
\usepackage{mathtools}
\usepackage{enumitem}
\usepackage[percent]{overpic}
\usepackage{tikz} 
\usepackage{mathrsfs}
\usepackage{xargs}
\usepackage[colorinlistoftodos,prependcaption,textsize=tiny]{todonotes}
\newcommandx{\at}[2][1=]{\todo[linecolor=red,backgroundcolor=red!25,bordercolor=red,#1]{#2}}

\usepackage{algorithm}
\usepackage{algorithmic}

\usepackage{mathtools,booktabs}

\usepackage{varwidth}
\usepackage{hyperref}
\usepackage{cleveref}
\usepackage{cite}

\title{FEAST for differential eigenvalue problems\thanks{Submitted to the editors \today.
\funding{This work is supported by National Science Foundation grant no.~1818757.}}}
\author{Andrew Horning\thanks{Center for Applied Mathematics, Cornell University, Ithaca, NY 14853. (\email{ajh326@cornell.edu})} \and Alex Townsend\thanks{Department of Mathematics, Cornell University, Ithaca, NY  14853. (\email{townsend@cornell.edu})}}
\headers{FEAST for differential eigenproblems}{Andrew Horning and Alex Townsend}
\begin{document}
\maketitle

\begin{abstract}
An operator analogue of the FEAST matrix eigensolver is developed to compute the discrete part of the spectrum of a differential operator in a region of interest in the complex plane. Unbounded search regions are handled with a novel rational filter for the right half-plane. If the differential operator is normal or self-adjoint, then the operator analogue preserves that structure and robustly computes eigenvalues to near machine precision accuracy. The algorithm is particularly adept at computing high-frequency modes of differential operators that possess self-adjoint structure with respect to weighted Hilbert spaces.
\end{abstract}



\section{Introduction}\label{sec:introduction}
In this paper, we consider differential eigenvalue problems posed on the interval $[-1,1]$, i.e.,
\begin{equation}
\label{eqn:differential eigenvalue problem}
\mathcal{L}u=\lambda u, \qquad u(\pm 1) = \cdots = u^{(N/2)}(\pm 1)=0.
\end{equation}
Here, $\mathcal{L}$ is a linear, ordinary differential operator of even order $N$. 
A complex number $\lambda$ and a function $u$ satisfying~\cref{eqn:differential eigenvalue problem} are called an eigenvalue and eigenfunction of $\mathcal{L}$, respectively.  We focus on computing the eigenvalues of $\mathcal{L}$ contained in a simply connected region $\Omega\subset\mathbb{C}$. Throughout, we assume that the boundary $\partial\Omega$ is a rectifiable, simple closed curve and that the spectrum $\lambda(\mathcal{L})$ of $\mathcal{L}$ is discrete, does not intersect $\partial\Omega$, and only finitely many eigenvalues counting multiplicities are in $\Omega$. To simplify discussion about the eigenfunctions of~\cref{eqn:differential eigenvalue problem}, we assume that there are eigenfunctions of $\mathcal{L}$ that form a basis for the invariant subspace of $\mathcal{L}$ associated with $\Omega$.

Since the development of the QR algorithm in the 1960s, the standard methods for solving~\cref{eqn:differential eigenvalue problem} have adopted a ``discretize-then-solve" paradigm. These algorithms first discretize $\mathcal{L}$ to obtain a finite matrix eigenvalue problem and then solve the matrix eigenvalue problem with algorithms from numerical linear algebra~\cite{dongarra1996chebyshev,gary1965computing,goodman1965numerical,orszag_1971}.
Motivated by mathematical software for highly adaptive computations with functions~\cite{driscoll2014chebfun}, we propose an alternative strategy: an algorithm that solves~\cref{eqn:differential eigenvalue problem} by directly manipulating $\mathcal{L}$ at the continuous level and only discretizes functions, not operators. By designing an eigensolver for $\mathcal{L}$ rather than intermediate discretizations, we are able to leverage spectrally accurate approximation schemes for functions while avoiding several pitfalls that plague spectral discretizations of~\cref{eqn:differential eigenvalue problem} (see~\cite{trefethen1987instability},~\cite[Ch.~2]{gheorghiu2014spectral}, and~\cite[Ch.~30]{trefethen2005spectra}). For this reason, we view our proposed algorithms as adopting a ``solve-then-discretize" paradigm. This paradigm has been applied to Krylov methods~\cite{gilles2018continuous}, iterative eigensolvers~\cite{ChebfunExample}, and contour integral projection eigensolvers~\cite{ChebfunExample2} for differential operators. Related techniques for computing with operators on infinite dimensional spaces have been proposed and studied in~\cite{hansen2009infinite,olver2014practical}.

As an example of the advantages of our methodology, consider the simplest possible differential eigenvalue problem given by
\begin{equation}
\label{eqn:oscillator}
-\frac{d^2u}{dx^2}=\lambda u, \qquad u(\pm 1)=0.
\end{equation}
The eigenvalues of~\cref{eqn:oscillator} are $\lambda_k=(k\pi/2)^2$, for $k\geq 1$, and are well-conditioned due to the fact that the eigenfunctions form a complete orthonormal set in the Hilbert space $L^2([-1,1])$~\cite[p.~382]{kato2013perturbation}.  However, spectral discretizations of~\cref{eqn:oscillator} lead to highly non-normal matrices with eigenvalues that are far more ill-conditioned than expected. Due to this ill-conditioning, the accuracy in the computed eigenvalues can be extremely variable and difficult to predict, ranging from a few digits to nearly full precision (see~\cref{fig:motivation}).
\begin{figure}[!tbp]
\label{fig:motivation}
  \centering
  \begin{minipage}[b]{0.45\textwidth}
    \begin{overpic}[width=\textwidth]{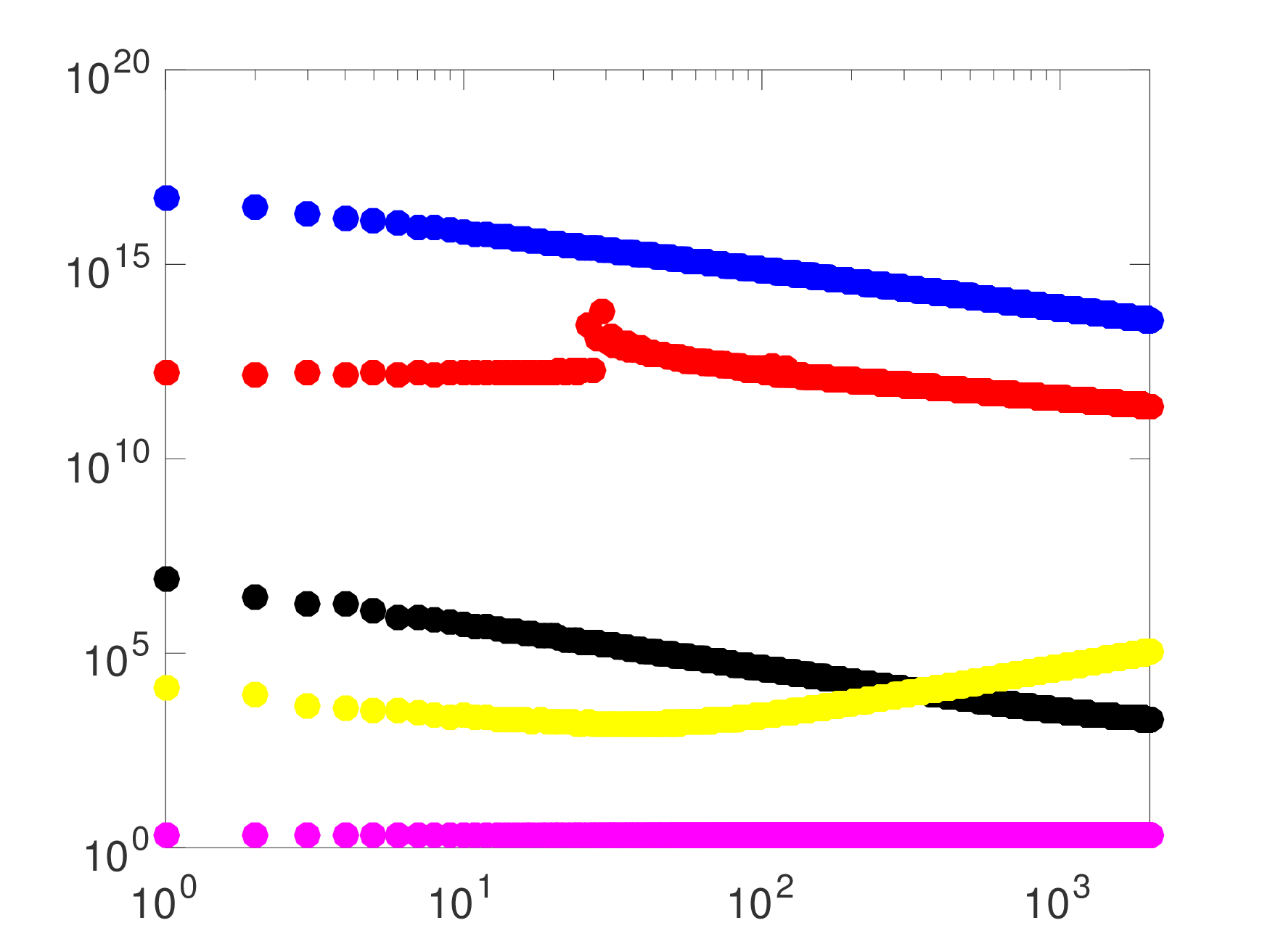}
 	\put (50,-2) {$\displaystyle n$}
 	\end{overpic}
  \end{minipage}
  \hfill
  \begin{minipage}[b]{0.45\textwidth}
    \begin{overpic}[width=\textwidth]{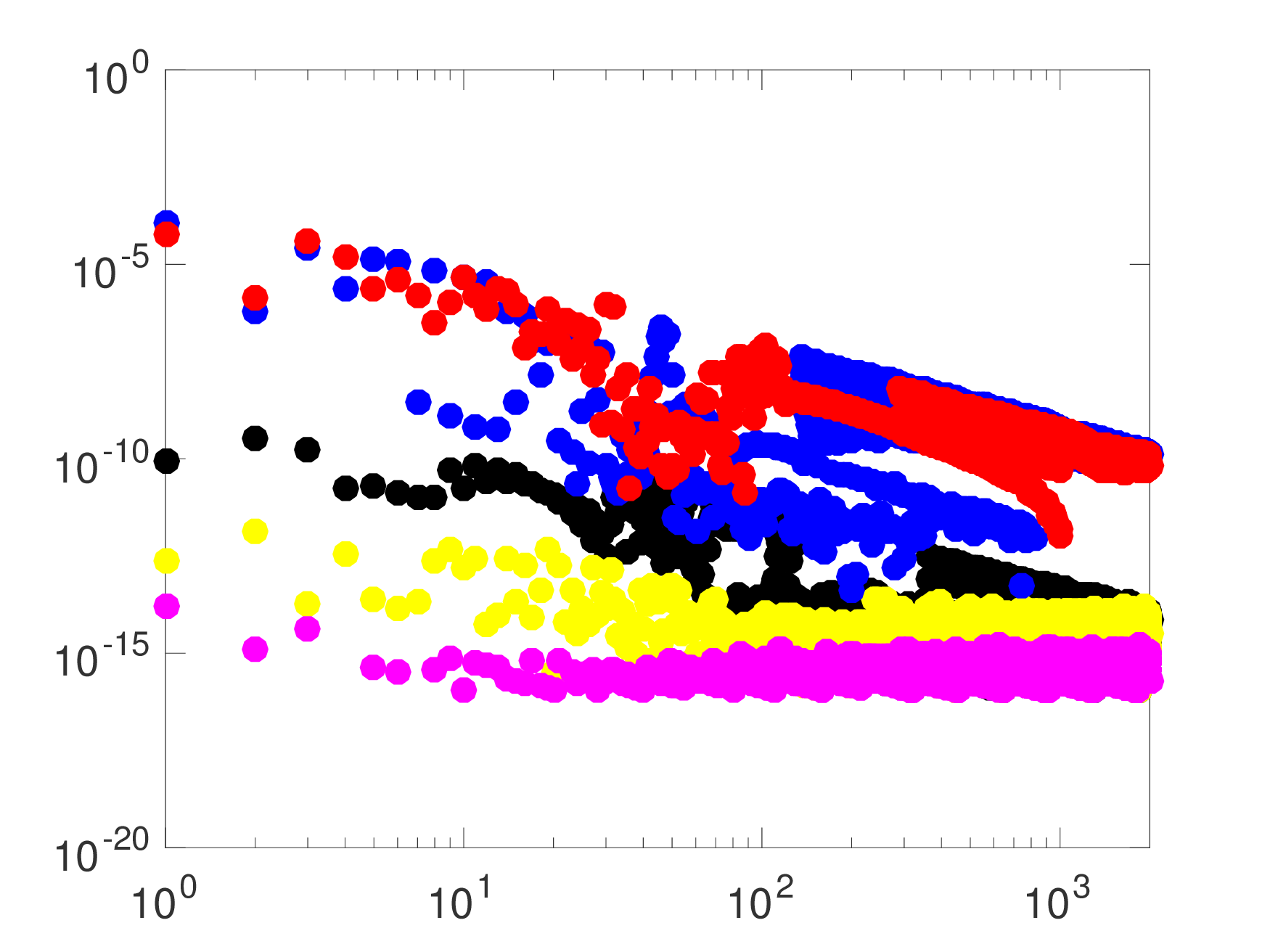}
 	\put (50,-2) {$\displaystyle n$}
 	\end{overpic}
  \end{minipage}
  \caption{Left: The eigenvalue condition numbers~\cite{ANGUAS2019170} for $4000\times 4000$ discretizations of~\cref{eqn:oscillator} obtained by collocation (blue dots), tau (red dots), Chebyshev--Galerkin (black dots), and ultraspherical (yellow dots) spectral methods are compared to the eigenvalue condition numbers (magenta dots) of~\cref{eqn:oscillator}, which are preserved by the operator analogue of FEAST. Right: The relative errors in the first $2000$ eigenvalues of each spectral discretization of~\cref{eqn:oscillator}, computed with a backward stable eigensolver~\cite[p.385]{golub2012matrix}. We observe fluctuations in the relative errors due to the ill-conditioning introduced by using nonsymmetric spectral discretizations of $\mathcal{L}$. In contrast, the relative errors (magenta dots) in the eigenvalues computed by \texttt{contFEAST}, a practical implementation of the operator analogue of FEAST (see~\cref{sec:practical algorithm}), are on the order of machine precision.}
\end{figure}
It is possible to use structure-preserving spectral discretizations to solve~\cref{eqn:oscillator} accurately~\cite{chen2008approximate,shen2007fourierization}. However, there is a lack of literature on designing spectral discretizations of~\cref{eqn:differential eigenvalue problem} when $\mathcal{L}$ is self-adjoint or normal with respect to an inner product other than $L^2([-1,1])$. On the other hand, our solve-then-discretize methodology automatically preserves the normality or self-adjointness of $\mathcal{L}$ with respect to a relevant Hilbert space $\mathcal{H}$, provided that the inner product $(\cdot,\cdot)_\mathcal{H}$ can be evaluated. 

At the heart of our approach is an operator analogue of the FEAST matrix eigensolver, which we briefly outline:
\begin{itemize}[leftmargin=*,noitemsep]
\item[(1)] We construct a basis for the eigenspace $\mathcal{V}$ corresponding to $\Omega$ by sampling the range of the associated spectral projector $\mathcal{P}_\mathcal{V}$.
\item[(2)] We extract an $\mathcal{H}$-orthonormal basis for $\mathcal{V}$ with a continuous analogue of the QR factorization~\cite{trefethen2009householder}.
\item[(3)] We perform a Rayleigh--Ritz projection~\cite[p.~98]{saad1992numerical} of $\mathcal{L}$ onto $\mathcal{V}$ with the orthonormal basis in (2). We solve the resulting matrix eigenvalue problem to obtain approximations to the eigenvalues of $\mathcal{L}$ in $\Omega$.
\end{itemize}

As with the FEAST matrix eigensolver, the spectral projector $\mathcal{P}_\mathcal{V}$ is applied approximately via a quadrature rule approximation. For matrices, this involves solving shifted linear systems, while for differential operators one needs to solve shifted linear differential equations. We solve these differential equations with the ultraspherical spectral method, which is a well-conditioned spectral method that is capable of resolving solutions that exhibit layers, rapid oscillations, and weak corner singularities~\cite{olver2013fast}.

Critically, we discretize basis functions for $\mathcal{V}$ as opposed to discretizing the differential operator $\mathcal{L}$ when solving~\cref{eqn:differential eigenvalue problem}.  While discretizations of a normal operator $\mathcal{L}$ can lead to non-normal matrices, the Rayleigh--Ritz projection described in (3) always leads to a normal matrix eigenvalue problem when $\mathcal{L}$ is normal (see~\cref{thm:condition number,thm:pseudospectra inclusion 1}). In fact, we prove that using a sufficiently good approximate basis for $\mathcal{V}$ does not significantly increase the sensitivity of the eigenvalues when $\mathcal{L}$ is normal (see~\cref{subsec:pseudospectral inclusion} for a precise statement). The result is a highly accurate eigensolver for normal differential operators $\mathcal{L}$, requiring $\mathcal{O}(mMN\log(N)+m^2N+m^3)$ floating point operations, where $m={\rm dim}(\mathcal{V})$ and $M$ and $N$ are the polynomial degrees used to resolve the variable coefficients in $\mathcal{L}$ and the eigenfunctions in $\mathcal{V}$, respectively.

The eigensolver we develop is competitive in the high-frequency regime because it efficiently resolves oscillatory basis functions in $\mathcal{V}$. Furthermore, it handles operators that are self-adjoint or normal with respect to non-standard Hilbert spaces. Finally, our algorithm is parallelizable like the FEAST matrix eigensolver~\cite{OriginalFEAST}. This work is a step towards closing the gap between the frequency regimes that are accessible to computational techniques and asymptotic methods for differential eigenvalue problems posed on higher dimensional domains~\cite{barnett2014fast,Resonances1D}.

The paper is organized as follows.  We begin in~\cref{sec:matrix FEAST} by reviewing FEAST for matrix eigenvalue problems. In~\cref{sec:contFEAST} we introduce an analogue of FEAST for differential operators and show that the operator analogue preserves eigenvalue sensitivity. In~\cref{sec:practical algorithm} we discuss a practical implementation of the operator analogue and provide two examples from Sturm--Liouville theory to illustrate its capabilities in the high-frequency regime. We analyze the convergence and stability of this implementation in~\cref{sec:convergence and stability}.~\Cref{sec:continuous RQI,sec:unbounded regions} develop further applications of the solve-then-discretize paradigm, including an operator analogue of the Rayleigh Quotient iteration and an extension of FEAST to unbounded search regions.

\section{The FEAST matrix eigensolver}\label{sec:matrix FEAST}
\begin{figure}[!tbp]
\label{fig:feast_setup}
  \centering
  \begin{minipage}[b]{0.45\textwidth}
    \begin{overpic}[width=\textwidth]{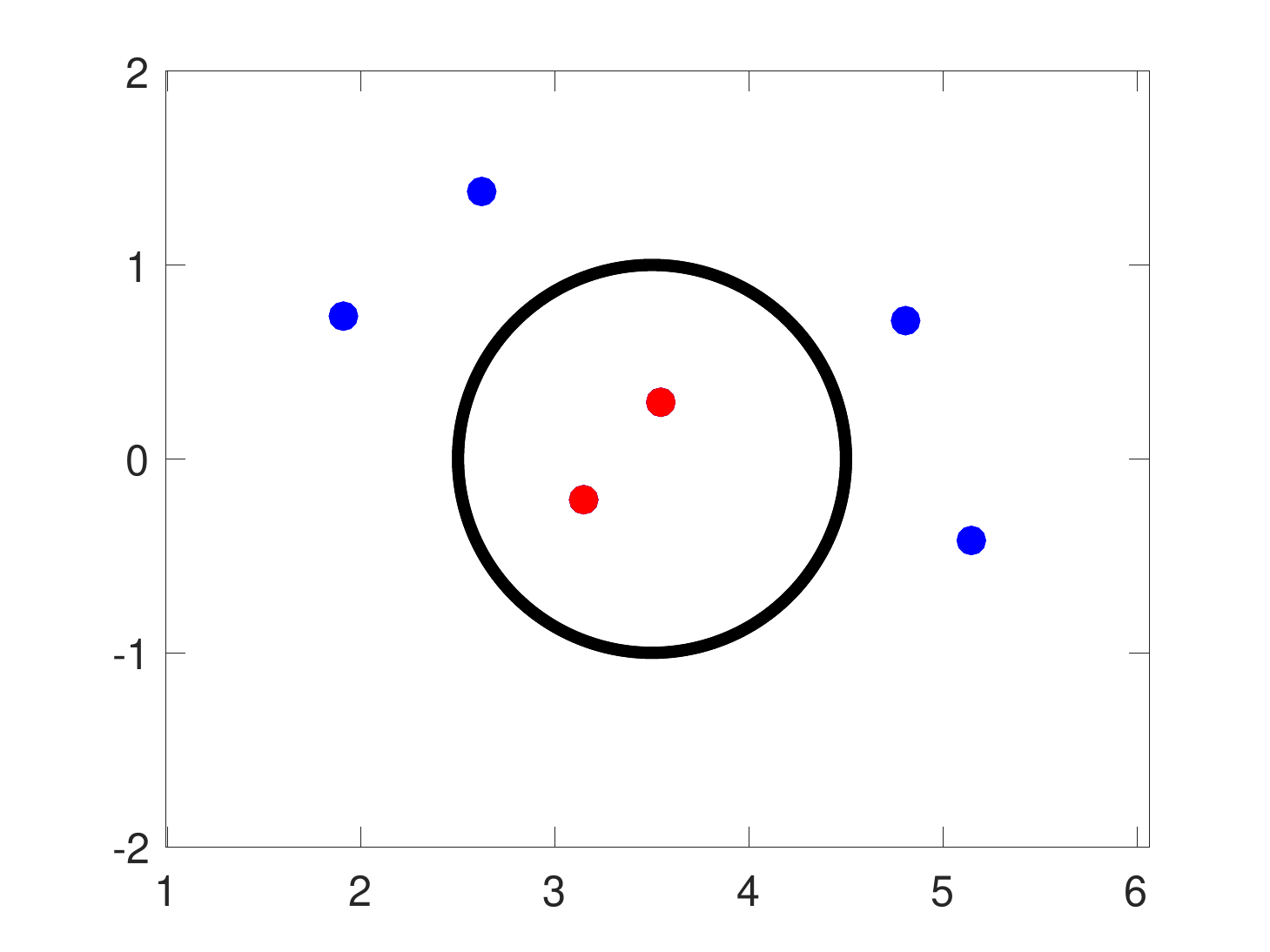}
 	\put (43,-3) {$\displaystyle{\rm Real}(z)$}
 	\put (0,28) {\rotatebox{90} {$\displaystyle{\rm Imag}(z)$}}
 	\put (53,31) {\Large$\displaystyle\Omega$}
	\end{overpic}
  \end{minipage}
  \hfill
  \begin{minipage}[b]{0.45\textwidth}
    \begin{overpic}[width=\textwidth]{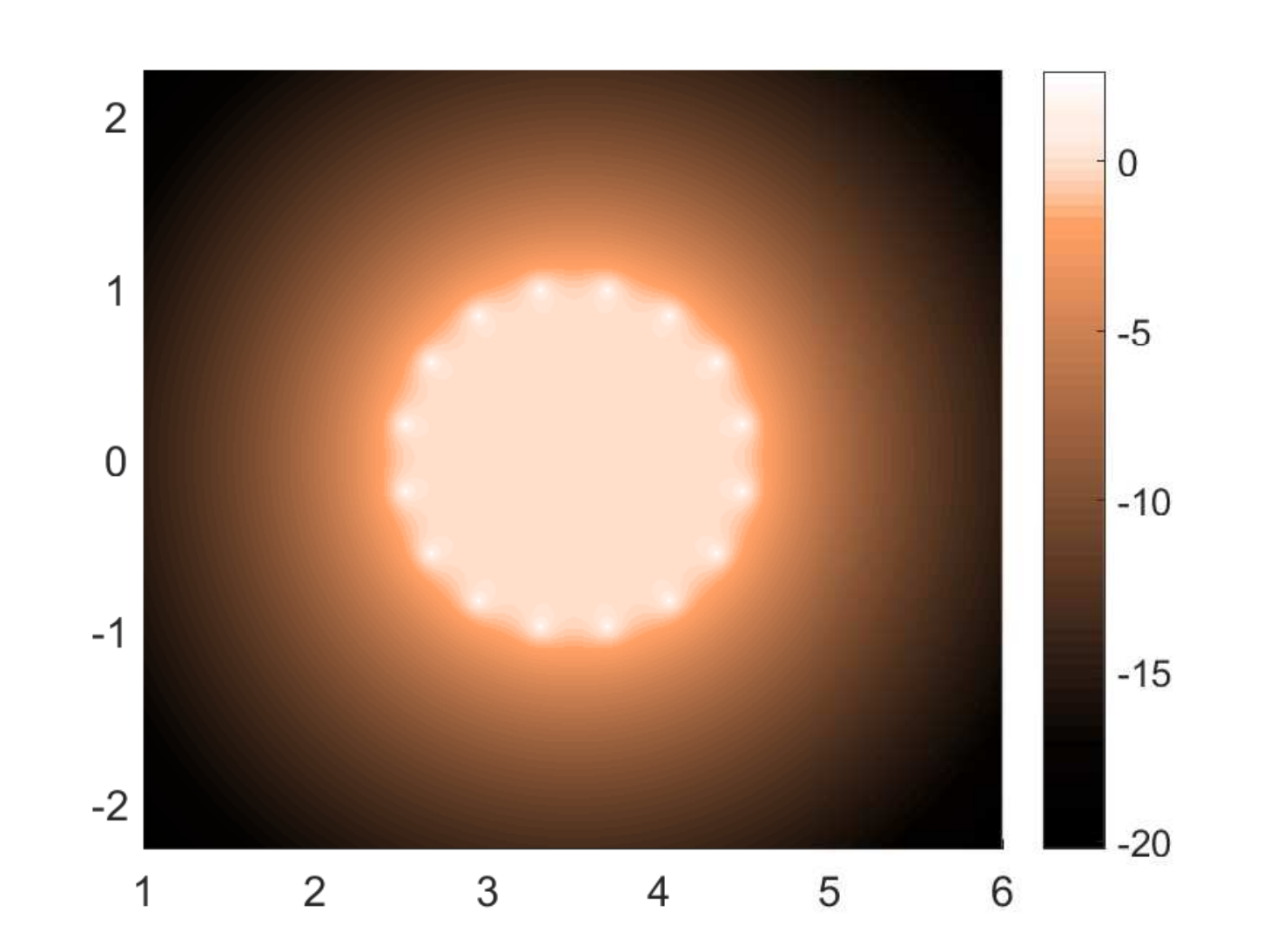}
 	\put (38,-3) {$\displaystyle{\rm Real}(z)$}
 	\put (-1,28) {\rotatebox{90} {$\displaystyle{\rm Imag}(z)$}}
	\end{overpic}
  \end{minipage}
  \caption{Left: FEAST uses an approximation to the spectral projector to compute the eigenvalues that lie inside $\Omega$ (red dots) and project away the eigenvalues outside of $\Omega$ (blue dots). Right: The rational map in~\cref{eqn:approx_FEAST_proj} that approximates the characteristic function on $\Omega$.}
\end{figure}
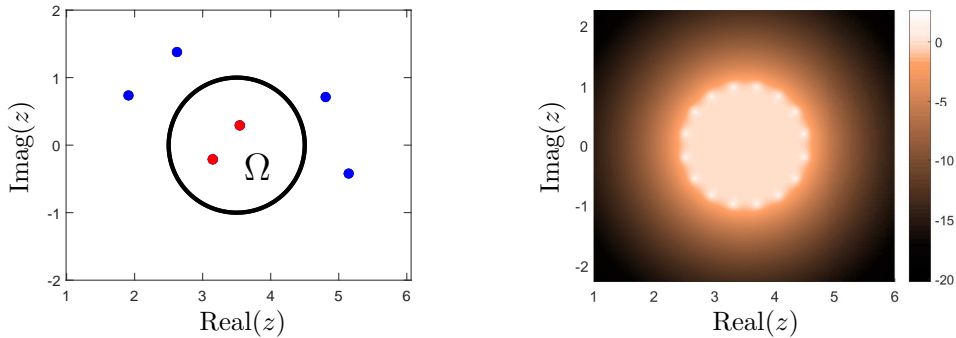
The FEAST matrix eigensolver uses approximate spectral projection to compute the eigenvalues of a matrix $A\in\mathbb{C}^{n\times n}$ in a region of interest $\Omega\subset\mathbb{C}$~\cite{kestyn2016feast} (see~\cref{fig:feast_setup}). It is usually more computationally efficient than standard eigensolvers when the number of eigenvalues in $\Omega$ is much smaller than $n$. The dominating computational cost of FEAST is solving several independent shifted linear systems, but these can be performed in parallel~\cite{kestyn2016feast}. 

There are three essential ingredients to FEAST:
\begin{itemize}[leftmargin=*,noitemsep]
\item[(i)]\textbf{Spectral projector.} 
Let $\lambda_1,\dots,\lambda_m$ be the eigenvalues of $A$ in $\Omega$ and let $\mathcal{V}$ be the associated invariant subspace of $A$, i.e., $A\mathcal{V}=\mathcal{V}$. 
The \textit{spectral projector} onto $\mathcal{V}$ is defined as
\begin{equation}\label{eqn:spectral projector}
P_\mathcal{V}=\frac{1}{2\pi i}\int_{\partial\Omega} (zI-A)^{-1}\, dz.
\end{equation}
The important fact here is that ${\rm range}(P_\mathcal{V})=\mathcal{V}$ and so $P_\mathcal{V}$ is a projection onto the invariant subspace of $A$~\cite{kato2013perturbation}.

\item[(ii)]\textbf{Basis for $\mathcal{V}$.}
FEAST uses the spectral projector to construct a basis for $\mathcal{V}$. It begins with a matrix $Y\in\mathbb{C}^{n\times m}$ with linearly independent columns that are not in ${\rm ker}(P_\mathcal{V})$, then it computes $Z=P_\mathcal{V}Y$. The columns of $Z$ span $\mathcal{V}$ and a QR factorization of $Z$ provides an orthonormal basis, $Q$, for $\mathcal{V}$.

\item[(iii)]
\textbf{Rayleigh--Ritz projection.} 
Having obtained an orthonormal basis for $\mathcal{V}$, FEAST solves $A_Qx=\lambda x$ using a dense eigensolver~\cite{OriginalFEAST}, where $A_Q=Q^*AQ$.  Since ${\rm range}(Q)=\mathcal{V}$, the eigenvalues of $A_Q$ are the eigenvalues of $A$ that lie inside $\Omega$. When $A_Q$ is diagonalizable, the eigenvectors of $A$ are given by $u_i=Qx_i$, for $i=1,\dots,m$, where $x_1,\dots, x_m$ are the corresponding eigenvectors of $A_Q$.
\end{itemize}

\begin{algorithm}[t]
\textbf{Input:} $A\in\mathbb{C}^{n\times n}$, $\Omega\subset \mathbb{C}$ containing $m$ eigenvalues of $A$, $Y:\mathbb{C}^{n\times m}$. \\
\vspace{-4mm}
\begin{algorithmic}[1]
\STATE Compute $V=\mathcal{P}_\mathcal{V}Y$.
\STATE Compute the QR factorization $V=QR$.
\STATE Compute $A_Q=Q^*AQ$ and solve the eigenvalue problem $A_QX=\Lambda X$ for $\Lambda={\rm diag}\left(\lambda_1,\dots,\lambda_m\right)$ and $X\in\mathbb{C}^{m\times m}$.
\end{algorithmic} \textbf{Output:} Eigenvalues $\lambda_1,\dots,\lambda_m$ in $\Omega$ and eigenfunctions $U=QX$.
\caption{The FEAST algorithm for matrix eigenvalue problems~\cite{OriginalFEAST}. This is often viewed as a single iteration that is repeated to improve the accuracy of the computed eigenvalues and eigenvectors~\cite{SubspaceIter}.}\label{alg:FEAST matrix eigensolver}
\end{algorithm}

For practical computation, FEAST approximates the contour integral in~\cref{eqn:spectral projector} with a quadrature rule. Given a quadrature rule with nodes $z_1,\ldots,z_\ell$ and weights $w_1,\ldots,w_\ell$, one can approximate $P_\mathcal{V}Y$ by
\begin{equation}
\label{eqn:approx_FEAST_proj}
P_\mathcal{V}Y\approx\frac{1}{2\pi i}\sum_{k=1}^\ell w_k(z_kI-A)^{-1}Y.
\end{equation}
In this case, the eigenpairs of $A_Q$ provide approximations to the eigenpairs of $A$, known as Ritz values and vectors~\cite{SubspaceIter}. 
To refine the accuracy of the Ritz values and vectors, a more accurate quadrature rule can be used to compute $P_\mathcal{V}Y$~\cite{SubspaceIter}. FEAST also refines the approximate eigenvalues and eigenvectors by applying $P_\mathcal{V}$ to the $n\times m$ block of approximate eigenvectors using a quadrature rule and iterating (ii) and (iii) until convergence. To fully understand this refinement process, one must examine FEAST through the lens of rational subspace iteration~\cite{SubspaceIter}.

When the dimension $m$ of the invariant subspace $\mathcal{V}$ is unknown, there are several techniques for estimating $m$ and selecting an appropriate value~\cite{kestyn2016feast,di2016efficient,SubspaceIter}. Most of these techniques can be incorporated into the operator analogue of FEAST. Consequently, we assume that $m$ is known throughout the paper and focus on the algorithmic and theoretical aspects of FEAST that are relevant in the operator setting.

Curiously, the originally proposed FEAST algorithm does not compute an orthonormal basis for $\mathcal{V}$ before performing the Rayleigh--Ritz projection~\cite{OriginalFEAST,kestyn2016feast}.\footnote{The FEAST algorithm for non-Hermitian matrices utilizes dual bases for the left and right eigenspaces to improve stability~\cite{kestyn2016feast}.} However, when $Q$ has orthonormal columns and ${\rm range}(Q)$ is an invariant subspace of $A$, then the eigenvalues of the small matrix $Q^*AQ$ are no more sensitive to perturbations than the original eigenvalues of $A$. This highly desirable property follows from an examination of the structure of the left and right invariant subspaces of $Q^*AQ$ or, alternatively, from the $\epsilon$-pseudospectra of $Q^*AQ$~\cite[p. 382]{trefethen2005spectra}. 

\section{An operator analogue of FEAST}\label{sec:contFEAST}
The FEAST matrix algorithm provides a natural starting point for an operator analogue because it provides a recipe to construct a small matrix $Q^*AQ$ whose eigenvalues coincide with those of $A$ inside $\Omega$ and have related invariant subspaces. Moreover, the eigenstructure of $Q^*AQ$ reflects the eigenstructure of $A$ when the columns of $Q$ are orthonormal. As the sensitivity of the eigenvalues of $A$ depends intimately on the structure of the associated eigenvectors, $Q^*AQ$ may be used to compute the desired eigenvalues of $A$ efficiently without sacrificing accuracy. Here, we generalize FEAST so that it constructs a matrix whose eigenvalues coincide with those of a differential operator inside $\Omega$.

\subsection{FEAST for differential operators}\label{subsec:theoretical framework}
In place of a matrix $A$ acting on vectors from $\mathbb{C}^n$, we now consider a differential operator $\mathcal{L}$ acting on functions from a Hilbert space $\mathcal{H}$. As described in~\cref{sec:matrix FEAST}, the FEAST recipe prescribes a spectral projection to compute a basis for $\mathcal{V}$, which is then used for the Rayleigh--Ritz projection to construct a matrix representation on $\mathcal{V}$. Throughout, we require that $\mathcal{L}$ be a closed operator\footnote{An operator $\mathcal{A}:\mathcal{D}(\mathcal{A})\rightarrow\mathcal{H}$ is \textit{closed} if its graph is a closed linear subspace of $\mathcal{H}\times\mathcal{H}$~\cite[p.~165]{kato2013perturbation}.} and that its domain $\mathcal{D}(\mathcal{L})$ is dense in $\mathcal{H}$.

\begin{itemize}[leftmargin=*,noitemsep]
\item[(i)]\textbf{Spectral projector.} 
Although $\mathcal{L}$ is unbounded, the resolvent $(z\mathcal{I-L})^{-1}$ is bounded when $z\not\in\lambda(\mathcal{L})$ and the spectral projector onto $\mathcal{V}$ may be defined via contour integral~\cite[p.~178]{kato2013perturbation}. It is given by
\begin{equation}
\label{eqn:spectral projector for diff op}
\mathcal{P}_\mathcal{V}=\frac{1}{2\pi i}\int_{\partial\Omega}(z\mathcal{I-L})^{-1}\, dz.
\end{equation}

\item[(ii)]\textbf{Basis for $\mathcal{V}$.} With the spectral projector at our disposal, we apply $\mathcal{P}_\mathcal{V}$ to functions $f_1,\dots,f_m$ in $\mathcal{H}\setminus{\rm ker}(\mathcal{P}_\mathcal{V})$ to obtain a basis of functions $v_1,\dots,v_m$ for $\mathcal{V}$. Orthonormalizing $v_1,\dots,v_m$ with respect to the inner product $(\cdot,\cdot)_\mathcal{H}$ on $\mathcal{H}$ gives us an $\mathcal{H}$-orthonormal basis $q_1,\dots,q_m$ for $\mathcal{V}$.

\item[(iii)]\textbf{Rayleigh--Ritz projection.} To compute a matrix representation $L$ of $\mathcal{L}$ on $\mathcal{V}$, the Rayleigh--Ritz projection is performed using the inner product on $\mathcal{H}$. The elements of $L$ are given by $L_{ij}=(q_i,\mathcal{L}q_j)_\mathcal{H}$ for $1\leq i,j \leq m$. The eigenvalues of $L$ are precisely the eigenvalues $\lambda_1,\dots,\lambda_m$ of $\mathcal{L}$ that lie inside $\Omega$. The eigenfunctions of $\mathcal{L}$ are recovered from the eigenvectors $x_1,\dots,x_m$ of $L$ by computing $u_i=\sum_{k=1}^m x_i^{(k)}q_k$, for $i=1,\dots,m$, where $x_i^{(k)}$ is the $k^{th}$ component of $x_i$.
\end{itemize}

To avoid a clutter of indices, we employ the notation of quasimatrices.\footnote{A \textit{quasimatrix} is a matrix whose columns (or rows) are functions defined on an interval $[a,b]$, in contrast to matrices whose columns (or rows) are vectors~\cite[Ch.6]{driscoll2014chebfun}.} If $Q$ is the quasimatrix with columns $q_1,\dots,q_m$, then the matrix $L$ whose elements are $L_{ij}=(q_i,\mathcal{L}q_j)_\mathcal{H}$ in (iii) is expressed compactly in quasimatrix notation as $L=Q^*\mathcal{L}Q$. Here, $Q^*$ is the conjugate transpose of the quasimatrix $Q$ so its rows are complex conjugates of the functions $q_1,\dots, q_m$.

The analogue of FEAST for differential operators is summarized in Algorithm~\ref{alg:continuous analogue} using quasimatrix notation so that it resembles its matrix counterpart.
\begin{algorithm}[t]
\textbf{Input:} $\mathcal{L}:\mathcal{D}(\mathcal{L})\rightarrow\mathcal{H}$, $\Omega\subset \mathbb{C}$  containing $m$ eigenvalues of $\mathcal{L}$, $F:\mathbb{C}^m\rightarrow\mathcal{H}$. \\
\vspace{-4mm}
\begin{algorithmic}[1]
\STATE Compute $V=\mathcal{P}_\mathcal{V}F$.
\STATE Compute $V=QR$, where $Q:\mathbb{C}^m\rightarrow\mathcal{D}(\mathcal{L})\subset\mathcal{H}$ has $\mathcal{H}$-orthonormal columns and $R\in\mathbb{C}^{m\times m}$ is upper triangular.
\STATE Compute $L=Q^*\mathcal{L}Q$ and solve $LX=\Lambda X$ for $\Lambda={\rm diag}[\lambda_1,\dots,\lambda_m]$ and $X\in\mathbb{C}^{m\times m}$.
\end{algorithmic} \textbf{Output:} Eigenvalues $\lambda_1,\dots,\lambda_m$ in $\Omega$ and eigenfunctions $U=QX$.
\caption{An operator analogue of FEAST for differential operators.}\label{alg:continuous analogue}
\end{algorithm}
Keep in mind that Algorithm~\ref{alg:continuous analogue} is a formal algorithm. In general, we cannot apply the spectral projector exactly, nor represent the basis $\mathcal{V}$ exactly with finite memory. A practical implementation is discussed in~\cref{sec:practical algorithm}.

\subsection{Condition number of the Ritz values}\label{subsec:Ritz values}
As illustrated in~\cref{fig:motivation}, the eigenvalues of matrix discretizations of $\mathcal{L}$ can be more sensitive to perturbations than the eigenvalues of $\mathcal{L}$. The advantage of our FEAST approach in~\cref{subsec:theoretical framework} is that the Ritz values, i.e., the eigenvalues of $Q^*\mathcal{L}Q$, are no more sensitive to perturbations than the original eigenvalues of $\mathcal{L}$ when ${\rm range}(Q)$ is an invariant subspace of $\mathcal{L}$.

To see this, let $\lambda$ be a simple eigenvalue of a differential operator $\mathcal{L}$. Let $u,w\in\mathcal{H}$ satisfy $\mathcal{L}u=\lambda u$ and $\mathcal{L}^*w=\overline{\lambda} w$, where $\overline{\lambda}$ denotes the complex conjugate of $\lambda$. The \textit{condition number}\footnote{Although this formula is usually associated with the condition number for a simple eigenvalue of a matrix, its proof extends to our general setting~\cite[Theorem 5]{tisseur2000backward}.} of $\lambda$ is given by~\cite[Theorem 2.3]{anguas2019comparison}
\begin{equation}
\label{eqn:eig condition number}
\kappa_\mathcal{H}(\lambda)=\frac{\lVert u\rVert_\mathcal{H}\lVert w\rVert_\mathcal{H}}{(w,u)_\mathcal{H}}.
\end{equation}
The condition number $\kappa_\mathcal{H}(\lambda)$ quantifies the worst-case first-order sensitivity of $\lambda$ to perturbations of $\mathcal{L}$.  For instance, if we compute $\lambda$ using a backward stable algorithm in floating point arithmetic, we expect to achieve an accuracy of at least $\kappa_\mathcal{H}(\lambda)\epsilon_{\text{mach}}$, where $\epsilon_{\text{mach}}$ is machine precision~\cite[Theorem~15.1]{trefethen1997numerical}.

\begin{theorem}\label{thm:condition number}
Let $\mathcal{L}:\mathcal{D}(\mathcal{L})\rightarrow\mathcal{H}$ be a closed and densely defined operator on a Hilbert space $\mathcal{H}$, $Q:\mathbb{C}^m\rightarrow\mathcal{H}$ be an invariant subspace of $\mathcal{L}$ satisfying $Q^*Q=I$, and $L=Q^*\mathcal{L}Q$. Suppose that $u\in{\rm range}(Q)$ satisfies $\mathcal{L}u=\lambda u$ and $w$ satisfies $\mathcal{L}^*w=\overline{\lambda} w$, where $\mathcal{L}^*$ denotes the adjoint of $\mathcal{L}$ and $\lambda$ is a simple eigenvalue of $\mathcal{L}$ with condition number $\kappa_\mathcal{H}(\lambda)$.  Then, 
\begin{itemize}
\item[1)] $LQ^*u=\lambda Q^*u$ and $L^*Q^*w=\overline{\lambda} Q^*w$,
\item[2)] $(Q^*w,Q^*u)_{\mathbb{C}^m}=(w,u)_\mathcal{H}$, and
\item[3)] $\kappa_{\mathbb{C}^m}(\lambda)\leq\kappa_\mathcal{H}(\lambda)$.
\end{itemize}
\end{theorem}
\begin{proof}
Denote $x=Q^*u$ and $y=Q^*w$. We prove the statements of the theorem in order. 1) Since $u\in {\rm range}(Q)$, we can write $u=Qx$. Then, $\mathcal{L}(Qx)=\lambda (Qx)$ implies that $Q^*\mathcal{L}Qx=\lambda x$ using the fact that $Q^*Q=I$. For the left eigenvector, we write $w=Qy+v$ for some $v\in {\rm range}(Q)^\perp$. Rewriting the adjoint equation for $w$, we find that $\mathcal{L}^*(Qy+v)=\overline{\lambda}(Qy+v)$ and multiplying by $Q^*$ on both sides yields $Q^*\mathcal{L}^*Qy=\overline{\lambda} y$. Here, we have used the fact that $Q^*\mathcal{L}^*v=0$, which holds because $v^*\mathcal{L}Q=0$. 2) By calculating $(w,u)_\mathcal{H}=(Qy+v,Qx)_\mathcal{H}$, we find that $(w,u)_\mathcal{H}=(Qy,Qx)_\mathcal{H}$ because $v\in{\rm range}(Q)^\perp$. Moreover, since $Q^*Q=I$ we conclude that $(w,u)_\mathcal{H}=(y,Q^*Qx)_{\mathbb{C}^m}=(y,x)_{\mathbb{C}^m}$. 3) We know that $\lVert u\rVert_\mathcal{H}=(Qx,Qx)_\mathcal{H}=(x,x)_{\mathbb{C}^m}=\lVert x\rVert_{\mathbb{C}^m}$ and $\lVert w\rVert_\mathcal{H}=(Qy+v,Qy+v)_\mathcal{H}=\lVert y\rVert_{\mathbb{C}^m}+\lVert v\rVert_\mathcal{H}$. Therefore, 
\[
\lVert u\rVert_\mathcal{H}\lVert w\rVert_\mathcal{H}=\lVert x\rVert_{\mathbb{C}^m}\left(\lVert y\rVert_{\mathbb{C}^m}+\lVert v\rVert_\mathcal{H}\right)\geq\lVert x\rVert_{\mathbb{C}^m}\lVert y\rVert_{\mathbb{C}^m}.
\] 
Referring to 2) for equality of the inner products in the denominator, we have 
\[
\kappa_{\mathbb{C}^m}(\lambda)=\frac{\lVert x\rVert_{\mathbb{C}^m}\lVert y\rVert_{\mathbb{C}^m}}{(y,x)_{\mathbb{C}^m}}\leq\frac{\lVert u\rVert_\mathcal{H}\lVert w\rVert_\mathcal{H}}{(w,u)_\mathcal{H}}=\kappa_\mathcal{H}(\lambda),
\] 
which concludes the proof.
\end{proof}

\Cref{thm:condition number} shows that if $\mathcal{L}$ is a normal operator, then $u=w$ and we have $\kappa_{\mathbb{C}^m}(\lambda)=\kappa_\mathcal{H}(\lambda)=1$. For non-normal operators, item 3) of~\cref{thm:condition number} may seem to erroneously indicate that ill-conditioning in the eigenvalues of $\mathcal{L}$ can be overcome by a Rayleigh--Ritz projection. However, when $\mathcal{L}$ is non-normal the spectral projector $\mathcal{P}_\mathcal{V}$ is an oblique projection and computing the basis $Q$ may be itself an ill-conditioned problem. \Cref{thm:condition number} also illustrates why the operator analogue of FEAST leads to a well-conditioned matrix eigenvalue problem when the differential eigenvalue problem is well-conditioned. By computing an $\mathcal{H}$-orthonormal basis for the Rayleigh--Ritz projection, the relevant structure in the eigenspaces of $\mathcal{L}$ and $\mathcal{L}^*$ is preserved. However, the first-order analysis above is limited to simple eigenvalues.

\subsection{Pseudospectra of $\mathbf{Q^*\mathcal{L}Q}$}
To go beyond first-order sensitivity analysis, we compare the $\epsilon$-pseudospectra of $\mathcal{L}$ and $Q^*\mathcal{L}Q$. Fix any $\epsilon>0$ and let $\mathcal{L}:\mathcal{D}(\mathcal{L})\rightarrow\mathcal{H}$ be a closed operator with a domain $\mathcal{D}(\mathcal{L})$ that is dense in $\mathcal{H}$. The $\epsilon$-pseudospectrum of $\mathcal{L}$ is defined as the set~\cite[p.~31]{trefethen2005spectra}
\begin{equation}
\label{eqn:epsilon pseudospectra}
\lambda_\epsilon(\mathcal{L})=\{z\in\mathbb{C}:\lVert(z\mathcal{I-L})^{-1}\rVert_\mathcal{H}>1/\epsilon\}.
\end{equation}
Here, we adopt the usual convention that $\lVert(z\mathcal{I-L})^{-1}\rVert_\mathcal{H}=\infty$ when $z\in\lambda(\mathcal{L})$ so that $\lambda(\mathcal{L})\subset\lambda_\epsilon(\mathcal{L})$. The $\epsilon$-pseudospectrum set of $\mathcal{L}$ bounds the region in which the eigenvalues of the perturbed operator $\mathcal{L+E}$ with $\lVert\mathcal{E}\rVert_\mathcal{H}<\epsilon$ can be found~\cite[p.~31]{trefethen2005spectra}. This means that $\lambda(\mathcal{L+E})\subset\lambda_\epsilon(\mathcal{L})$. In fact, there is an equivalence so that~\cite[p.~31]{trefethen2005spectra} 
\begin{equation}
\label{eqn:pseudospectra and perturbations}
\bigcup_{\lVert\mathcal{E}\rVert_\mathcal{H}<\epsilon}\lambda(\mathcal{L+E})=\lambda_\epsilon(\mathcal{L}).
\end{equation} This allows us to relate the sensitivity of the eigenvalues of $\mathcal{L}$ and $Q^*\mathcal{L}Q$ by comparing the resolvent norms $\lVert(z\mathcal{I-L})^{-1}\rVert_\mathcal{H}$ and $\lVert(zI-Q^*\mathcal{L}Q)^{-1}\rVert_{\mathbb{C}^m}$, respectively.

A useful generalization of~\cref{thm:condition number} is that the $\epsilon$-pseudospectrum of $Q^*\mathcal{L}Q$ is contained in the $\epsilon$-pseudospectrum of $\mathcal{L}$. Since this holds for any $\epsilon>0$, it demonstrates that the eigenvalues (even those with multiplicity) of $Q^*\mathcal{L}Q$ are no more sensitive to perturbations than those of $\mathcal{L}$. This inclusion result is well-known in the matrix case where projection methods are a popular method for approximating the $\epsilon$-pseudospectra of large data-sparse matrices~\cite[p.~381]{trefethen2005spectra}.

\begin{theorem}
\label{thm:pseudospectra inclusion 1}
Let $\mathcal{L}:\mathcal{D}(\mathcal{L})\rightarrow\mathcal{H}$ be a closed and densely defined operator on a Hilbert space $\mathcal{H}$. For a fixed $\epsilon>0$, suppose that $Q:\mathbb{C}^m\rightarrow\mathcal{H}$ satisfies $Q^*Q=I$ and that ${\rm range}(Q)$ is an invariant subspace of $\mathcal{L}$. Then, $\lambda_\epsilon(Q^*\mathcal{L}Q)\subset\lambda_\epsilon(\mathcal{L})$.
\end{theorem}
\begin{proof}
We follow the proof of Proposition~40.1 in~\cite[p.~382]{trefethen2005spectra} for matrices, but with a closed operator. Since $Qx\in\mathcal{H}$ for any $x\in\mathbb{C}^m$, we have that
\begin{equation*}
\begin{aligned}
\lVert(zI-\mathcal{L})^{-1}\rVert_\mathcal{H}=\sup_{f\in\mathcal{H},\lVert f\rVert_\mathcal{H}=1}\lVert(z\mathcal{I-L})^{-1}f\rVert_\mathcal{H}
\geq \max_{x\in\mathbb{C}^m,\lVert x\rVert_{\mathbb{C}^m}=1}\lVert(z\mathcal{I-L})^{-1}Qx\rVert_\mathcal{H}.
\end{aligned}
\end{equation*}
Now, when ${\rm range}(Q)$ is an invariant subspace of $\mathcal{L}$ and $Q^*Q=I$, we have that $\lVert(z\mathcal{I-L})^{-1}Qx\rVert_{\mathcal{H}}=\lVert Q^*(z\mathcal{I-L})^{-1}Qx\rVert_{\mathbb{C}^m}$. Because $QQ^*f=f$ for all $f\in\mathcal{V}$, we can check that $Q^*(z\mathcal{I-L})^{-1}Q=(Q^*(z\mathcal{I-L})Q)^{-1}.$ Since $Q^*Q=I$, it follows that $$\lVert(zI-\mathcal{L})^{-1}\rVert_\mathcal{H}\geq \lVert(Q^*(z\mathcal{I-L})Q)^{-1}\rVert_{\mathbb{C}^m}=\lVert(zI-Q^*\mathcal{L}Q)^{-1}\rVert_{\mathbb{C}^m}.$$ Therefore, $z\in\lambda_\epsilon(\mathcal{L})$ whenever $z\in\lambda_\epsilon(Q^*\mathcal{L}Q)$.
\end{proof}

The inclusion in~\cref{thm:pseudospectra inclusion 1} may be strict, indicating that the eigenvalues of $Q^*\mathcal{L}Q$ are less sensitive than those of $\mathcal{L}$. For example, this may occur when the projection onto ${\rm range}(Q)$ targets a subset of well-conditioned eigenvalues of $\mathcal{L}$. However, we emphasize that ill-conditioning in the eigenvalues of $\mathcal{L}$ cannot be overcome by a Rayleigh--Ritz projection: in general, the situation is complicated~\cite[Ch. 40]{trefethen2005spectra}. 

\Cref{thm:pseudospectra inclusion 1} is useful for studying the stability of Algorithm~\ref{alg:continuous analogue}. If an approximate eigenvalue $\hat\lambda$ of $Q^*\mathcal{L}Q$ is computed with an error tolerance of $\epsilon>0$, then 
\[
\hat\lambda\in\lambda_\epsilon(Q^*\mathcal{L}Q)\subset\lambda_\epsilon(\mathcal{L}).
\] 
From this, we know by~\cref{eqn:pseudospectra and perturbations} that $\hat\lambda$ is an eigenvalue of a perturbed operator $\mathcal{L+E}$ with $\lVert\mathcal{E}\rVert_\mathcal{H}<\epsilon$. In other words, the operator analogue of FEAST, Algorithm~\ref{alg:continuous analogue}, is backward stable. As we see in~\cref{sec:convergence and stability},~\cref{thm:pseudospectra inclusion 1} is also the starting point for a stability analysis when the spectral projection is no longer exact and the Rayleigh--Ritz projection is performed with a matrix $\hat Q$ that only approximates a basis for an invariant subspace of $\mathcal{L}$.

\section{A practical differential eigensolver based on an operator analogue of FEAST}\label{sec:practical algorithm}
The operator analogue of FEAST requires the manipulation of objects such as differential operators, functions, and contour integrals (see Algorithm~\ref{alg:continuous analogue}). For a practical implementation, these objects must be discretized; however, we avoid discretizing $\mathcal{L}$ directly. Instead, we construct polynomial approximants to the basis for $\mathcal{V}$ by approximately solving shifted linear ODEs. These polynomial approximants are used in the Rayleigh--Ritz projection to compute the eigenvalues of $\mathcal{L}$ in $\Omega$.

Let $z_1,\dots,z_\ell$ and $w_1,\dots,w_\ell$ be a set of quadrature nodes and weights to approximate the integral in~\cref{eqn:spectral projector for diff op}. As FEAST does in the matrix case, we approximate $\mathcal{P}_\mathcal{V}$ in~\cref{eqn:spectral projector for diff op} with a quadrature rule as follows:
\begin{equation}
\label{eqn:approximate spectral projector}
\mathcal{\hat P}_\mathcal{V}=\frac{1}{2\pi i}\sum_{k=1}^\ell w_k(z_k\mathcal{I-L})^{-1}.
\end{equation}
If $F$ is a quasimatrix with columns $f_1,\dots,f_m\in\mathcal{H}$, then $\mathcal{P}_\mathcal{V}F$ is replaced by the approximation $\mathcal{\hat P}_\mathcal{V}F=\frac{1}{2\pi i}\sum_{k=1}^\ell w_k(z_k\mathcal{I-L})^{-1}F$. Therefore, to compute $\mathcal{\hat P}_\mathcal{V}F$ we need to solve $\ell$ shifted linear ODEs, each with $m$ righthand sides, i.e.,
\begin{equation}
\label{eqn:shifted linear ODEs}
(z_k\mathcal{I-L})g_{i,k}=f_i, \quad g_{i,k}(\pm 1)=\cdots =g_{i,k}^{(N/2)}(\pm 1)=0,\qquad  1\leq i\leq m.
\end{equation}
If the quasimatrix with columns $g_{1,k},\dots,g_{m,k}$ is denoted by $G_k$ for $k=1,\dots,\ell$, then we have $\mathcal{\hat P}_\mathcal{V}F=\sum_{k=1}^\ell w_kG_k$.

To construct a basis for $\mathcal{V}$, it is important to choose $F$ so that the columns of $\hat V=\mathcal{\hat P}_\mathcal{V}F$ are linearly independent and, if possible, well-conditioned. In analogy with the implementation of matrix FEAST~\cite{OriginalFEAST,kestyn2016feast}, we obtain the columns of $F$ by selecting $m$ band-limited random functions\footnote{A periodic band-limited random function on $[-L,L]$ is a periodic function defined by a truncated Fourier series with random (e.g.~standard Gaussian distributed) coefficients. In the non-periodic setting, the Fourier series is defined on a larger interval $[-L',L']$ and the domain is then truncated~\cite{filip2018smooth}.} on $[-1,1]$~\cite{filip2018smooth}. When $\mathcal{L}$ is a normal operator, this typically yields a well-conditioned basis $\hat V$. 

\begin{algorithm}[t]
\textbf{Input:} $\mathcal{L}:\mathcal{D}(\mathcal{L})\rightarrow\mathcal{H}$, $z_1,\dots,z_\ell\in\partial\Omega$, $w_1,\dots,w_\ell\in\mathbb{C}$, $F:\mathbb{C}^m\rightarrow\mathcal{H}$, $\epsilon>0$. \\
\vspace{-4mm}
\begin{algorithmic}[1]
\REPEAT
	\STATE Solve $(z_k\mathcal{I-L})G_k=F$, $G_k(\pm 1)=0,\dots,G_k^{(N/2)}(\pm 1)=0$, for $k=1,\dots,\ell$.
	\STATE Set $\hat V=\sum_{k=1}^\ell w_kG_k$.
	\STATE Compute $\hat V=\hat Q\hat R$, where $\hat Q:\mathbb{C}^m\rightarrow\mathcal{D}(\mathcal{L})\subset\mathcal{H}$ has $\mathcal{H}$-orthonormal columns and $\hat R\in\mathbb{C}^{m\times m}$ is upper triangular.
	\STATE Compute $\hat L=\hat Q^*\mathcal{L}\hat Q$ and solve $\hat L\hat X=\hat X\hat \Lambda$ for $\hat \Lambda={\rm diag}[\hat\lambda_1,\dots,\hat\lambda_m]$ and $\hat X\in\mathbb{C}^{m\times m}$. Set $F=\hat Q\hat X$.
\UNTIL{$\lVert\mathcal{L}F-F\hat\Lambda\rVert_\mathcal{H}\leq\epsilon\lVert\hat\Lambda\rVert_{\mathbb{C}^m}$.}
\end{algorithmic} \textbf{Output:} $\hat\Lambda$, $\hat U=\hat Q\hat X$.
\caption{A practical algorithm for computing the eigenvalues of a differential operator $\mathcal{L}$, which we refer to as \texttt{contFEAST}.}\label{alg:contFEAST}
\end{algorithm}
We now outline the key implementation details of our differential eigensolver:
\begin{itemize}[leftmargin=*,noitemsep]
\item[(i)]
\textbf{Approximate spectral projection.}
To compute $\hat V=\mathcal{\hat P}_\mathcal{V}F$, we solve the shifted linear ODEs in~\cref{eqn:shifted linear ODEs} using the ultraspherical spectral method~\cite{olver2013fast}. 
The ultraspherical spectral method leads to well-conditioned linear systems and is capable of accurately resolving the functions $g_{i,k}$ even when they are highly oscillatory or have boundary layers.  Moreover, an adaptive QR factorization automatically determines the degree of the polynomial interpolants needed to approximate the functions $g_{i,k}$ to near machine precision~\cite{ApproxFun,olver2014practical}. After accurately resolving the functions $g_{i,k}$, we can accurately compute a basis for $\mathcal{V}$ provided that both the spectral projector is well-conditioned (i.e., $\mathcal{L}$ is not highly non-normal) and the quadrature rule is sufficiently accurate.

\item[(ii)]
\textbf{Orthonormal basis.}
To compute an orthonormal basis $\hat Q$ for the columns of $\hat V$, we compute a QR factorization of the quasimatrix $\hat V$ by Householder triangularization~\cite{trefethen2009householder}. The Householder reflectors are constructed with respect to the inner product $(\cdot,\cdot)_\mathcal{H}$ so that the columns of $\hat Q$ are 
$\mathcal{H}$-orthonormal. 

\item[(iii)]
\textbf{Computing $\mathbf{\hat Q^*\mathcal{L}\hat Q}$.}
To construct the matrix $\hat L=\hat Q^*\mathcal{L}\hat Q$, we apply $\mathcal{L}$ to the columns of $\hat Q$ and then evaluate the action of $\hat Q^*$ on $\mathcal{L}\hat Q$. Multiplying $\hat Q^*$ with $\mathcal{L}\hat Q$ involves taking the inner products
\begin{equation}
\label{eqn:moment matrix}
\hat L_{ij}=(\hat q_i,\mathcal{L}\hat q_j)_\mathcal{H}, \qquad 1\leq i,j\leq m,
\end{equation}
where $\hat q_i$ denotes the $i$th column of $\hat Q$. The eigenvalues $\hat\lambda_1,\dots,\hat\lambda_m$ and eigenvectors $\hat x_1,\dots,\hat x_m$ of the matrix $\hat L$ are computed using the QR algorithm~\cite[p.~385]{golub2012matrix}.
\end{itemize}

Critically, the inner product $(\cdot,\cdot)_\mathcal{H}$ used in the QR factorization of $\hat V$ and the construction of $\hat Q^*\mathcal{L}\hat Q$ depends on the choice of the Hilbert space $\mathcal{H}$. As long as we are able to evaluate $(\cdot,\cdot)_\mathcal{H}$, we can exploit the fact that $\mathcal{L}$ is self-adjoint or a normal operator with respect to $(\cdot,\cdot)_\mathcal{H}$ so that we can accurately compute the eigenvalues of $\mathcal{L}$ in $\Omega$ (see~\cref{thm:pseudospectra inclusion 2}). For this reason, our algorithm is able to accurately compute the eigenvalues and eigenfunctions of differential operators that are self-adjoint with respect to non-standard Hilbert spaces (see~\cref{subsec:high-frequency eigenmodes}).

Evaluating the inner product $(\cdot,\cdot)_\mathcal{H}$ usually means computing an integral, which we approximate with a quadrature rule. For example, if $\mathcal{H}=L^2([-1,1])$,
\[
(f,g)_{L^2([-1,1])}=\int_{-1}^1\overline{f(x)} g(x)\, dx.
\]
Given the Gauss--Legendre quadrature nodes $x_1,\dots,x_p$ and weights $w_1,\dots,w_p$ on $[-1,1]$, then one uses the approximation~\cite{bogaert2014iteration}
\[
(f,g)_{L^2([-1,1])}\approx\sum_{k=1}^p w_k \overline{f(x_k)}g(x_k).
\]

A practical implementation of the operator analogue of FEAST is presented in Algorithm~\ref{alg:contFEAST}. As with matrix FEAST, there are two approaches for improving the accuracy of the Ritz values $\hat\lambda_1,\dots,\hat\lambda_m$ and vectors $\hat Q\hat x_1,\dots,\hat Q\hat x_m$. The first is to improve the accuracy of the quadrature rule in~\cref{eqn:approximate spectral projector}. The second is to iterate the algorithm by replacing $F$ by the quasimatrix $\hat U$ with columns $\hat u_i=\hat Q\hat x_i$ for $1\leq i\leq m$, repeating the process if necessary.\footnote{When $\mathcal{L}$ is non-normal the Ritz vectors $\hat Q\hat x_1,\dots,\hat Q\hat x_m$ may become numerically linearly dependent, which can lead to an ill-conditioned basis $\hat V$ in susbequent iterations. The robustness of Algorithm~\ref{alg:contFEAST} may be improved by computing the Schur vectors $v_1,\dots,v_m$ of $\hat L$ and using the orthonormal basis $\hat Q\hat v_1,\dots,\hat Q\hat v_m$ to seed the next iteration~\cite{stewart1976simultaneous}.} For normal operators, this iteration generates a sequence of quasimatrices $\hat Q_k$ with $\mathcal{H}$-orthonormal columns that converge to an $\mathcal{H}$-orthonormal basis for the invariant subspace $\mathcal{V}$ as $k\rightarrow\infty$. This can be viewed as a rational subspace iteration and geometric convergence of the Ritz pairs is typical (see~\cref{sec:convergence and stability}).

With either refinement strategy, the accuracy of the Ritz pairs may be monitored using the residual norm (see step~6 of Algorithm~\ref{alg:contFEAST}) as a proxy, just as in the matrix case. For normal operators, the error in the eigenvalues and eigenvectors computed by Algorithm~\ref{alg:contFEAST} is typically $\mathcal{O}(\epsilon)$, where $\epsilon$ is the threshold for the residual norm in step~6. We defer a discussion of the convergence and stability of Algorithm~\ref{alg:contFEAST} to~\cref{sec:convergence and stability}. Additional resources on residual norm bounds for eigenvalues and eigenvectors of matrices and extensions to closed linear operators are found in~\cite{bai2000templates,chatelin2011spectral,stewart1971error}.

In practice, when $\mathcal{L}$ is non-normal it may be beneficial to use a dual Rayleigh--Ritz projection $\hat Q_L^*\mathcal{L}\hat Q_R$, where the columns of $\hat Q_R$ approximate an orthonormal basis for the target eigenspace of $\mathcal{L}$ and the columns of $\hat Q_L$ approximate an orthonormal basis for the associated eigenspace of the adjoint $\mathcal{L}^*$. In the case of matrix FEAST, the use of the dual projection leads to a non-normal matrix eigensolver with improved robustness~\cite{kestyn2016feast}. Although it is not difficult to adapt Algorithm~\ref{alg:contFEAST} to an operator analogue of FEAST that uses dual projection, we focus on the implementation and analysis of the one-sided iteration.

Typically, solving the ODEs in \eqref{eqn:shifted linear ODEs} dominates the computational cost of Algorithm~\ref{alg:contFEAST}. With the ultraspherical spectral method, the computational complexity of solving the linear ODEs with $m$ distinct right hand sides is $\mathcal{O}(mMN\log(N))$ floating point operations (flops)~\cite{olver2013fast}. Here, $N$ and $M$ are, respectively, the degrees of the truncated Chebyshev series needed to resolve the columns of $G_k$ and the variable coefficients in $\mathcal{L}$ to within the tolerance $\epsilon$ specified in Algorithm~\ref{alg:contFEAST}. In addition to the ODE solve, the QR factorization in (ii) requires $\mathcal{O}(m^2N)$ flops~\cite{trefethen2009householder}, while the dense eigenvalue computation with a small $m\times m$ matrix in (iii) takes $\mathcal{O}(m^3)$ flops~\cite[p.~391]{golub2012matrix}. The complexity of one iteration of Algorithm~\ref{alg:contFEAST} is therefore $\mathcal{O}(mMN\log(N)+m^2N+m^3)$ flops. In practice, convergence to machine precision usually occurs within two or three iterations.

\subsection{Computing high-frequency eigenmodes}\label{subsec:high-frequency eigenmodes}
Algorithm~\ref{alg:contFEAST} adaptively and accurately resolves basis functions for highly oscillatory eigenmodes and preserves the sensitivity of the eigenvalues of the differential operator $\mathcal{L}$, so it is well-suited to computing high-frequency eigenmodes when $\mathcal{L}$ is self-adjoint or normal with respect to $(\cdot,\cdot)_\mathcal{H}$. We provide two examples from Sturm--Liouville theory to illustrate the effectiveness of the solve-then-discretize methodology in the high-frequency regime.

\subsubsection{A regular Sturm--Liouville eigenvalue problem}\label{subsec:regular SLEP}
First consider a regular Sturm--Liouville eigenvalue problem (SLEP) given by 
\begin{equation}
\label{eqn:regular SLEP}
-\frac{d^2u}{dx^2}+x^2u=\lambda\cosh(x) u, \qquad u(\pm 1)=0.
\end{equation}
This defines a self-adjoint differential operator with respect to the inner product
\begin{equation}
\label{eqn:IP for regular SLEP}
(v,u)_w=\int_{-1}^1 \overline{v}u\cosh(x)\, dx.
\end{equation} 
Consequently,~\cref{eqn:regular SLEP} possesses a complete $(\cdot,\cdot)_w$-orthonormal basis of eigenfunctions $u_1,u_2,u_3,\ldots$ for the weighted Hilbert space $\mathcal{H}_w=\{u : \lVert u\rVert_w=\sqrt{(u,u)_w}<\infty\}$ and an unbounded set of real eigenvalues $\lambda_1\leq \lambda_2\leq \lambda_3\leq \cdots$.

Asymptotics for the large eigenvalues of~\cref{eqn:regular SLEP} are given by~\cite{akbarfam2004higher}
\begin{equation}
\label{eqn:SLEP asymptotics}
\sqrt{\lambda_n}\sim \frac{n\pi}{\int_{-1}^1\sqrt{\cosh(x)}\, dx}, \qquad n\rightarrow\infty.
\end{equation}
To accurately compute the large eigenvalues of~\cref{eqn:regular SLEP} with Algorithm~\ref{alg:contFEAST}, we prescribe circular search regions with unit radius centered at the values given by the asymptotic formula in~\cref{eqn:SLEP asymptotics} (see~\cref{fig:regular SLEP}). Each search region contains one eigenvalue.
\begin{figure}[!tbp]
\label{fig:regular SLEP}
  \centering
  \begin{minipage}[b]{0.45\textwidth}
    \begin{overpic}[width=\textwidth]{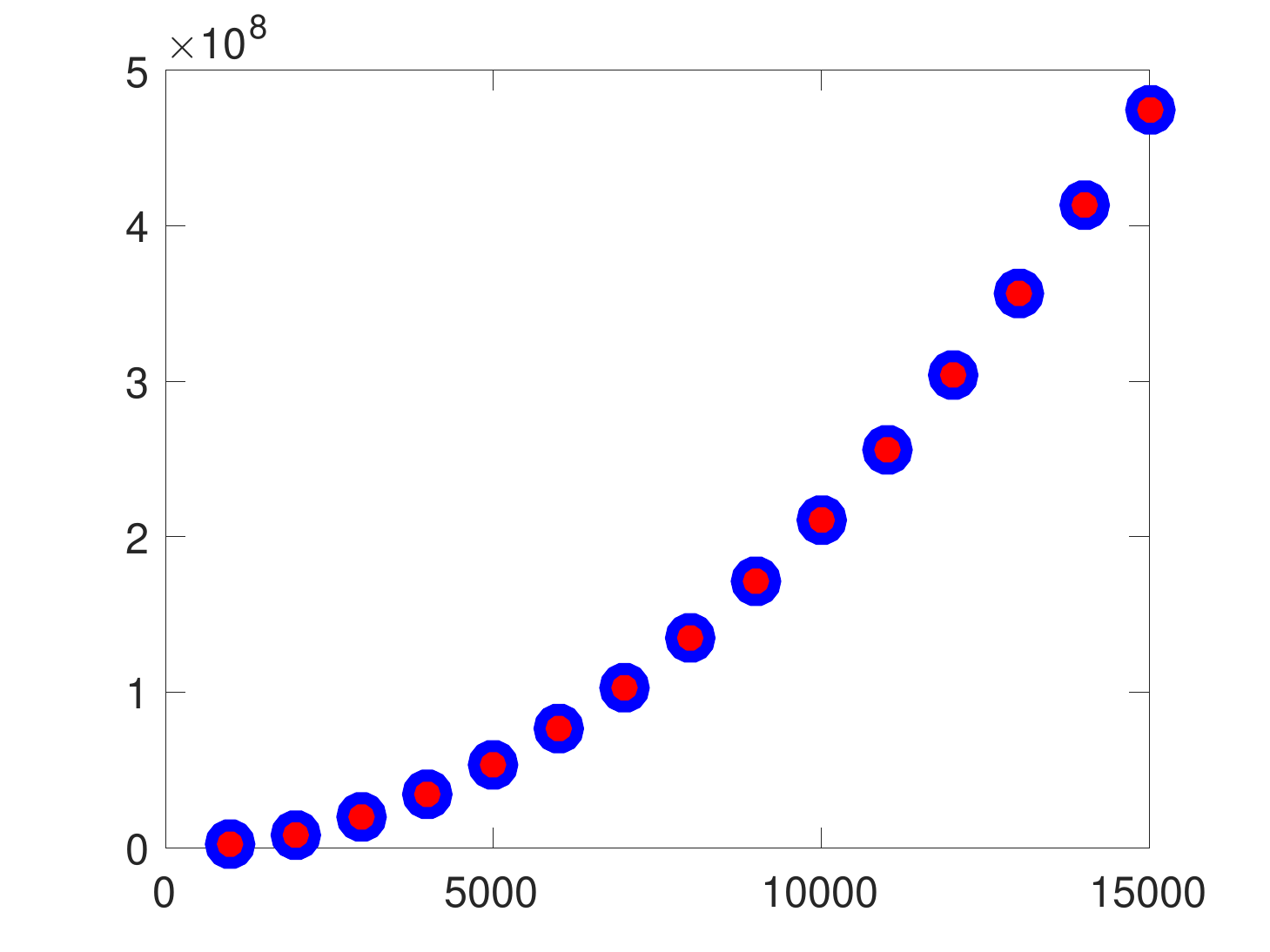}
 	\put (50,-2) {$\displaystyle n$}
 	\put(48,72) {$\hat\lambda_n$}
	\end{overpic}
  \end{minipage}
  \hfill
  \begin{minipage}[b]{0.45\textwidth}
    \begin{overpic}[width=\textwidth]{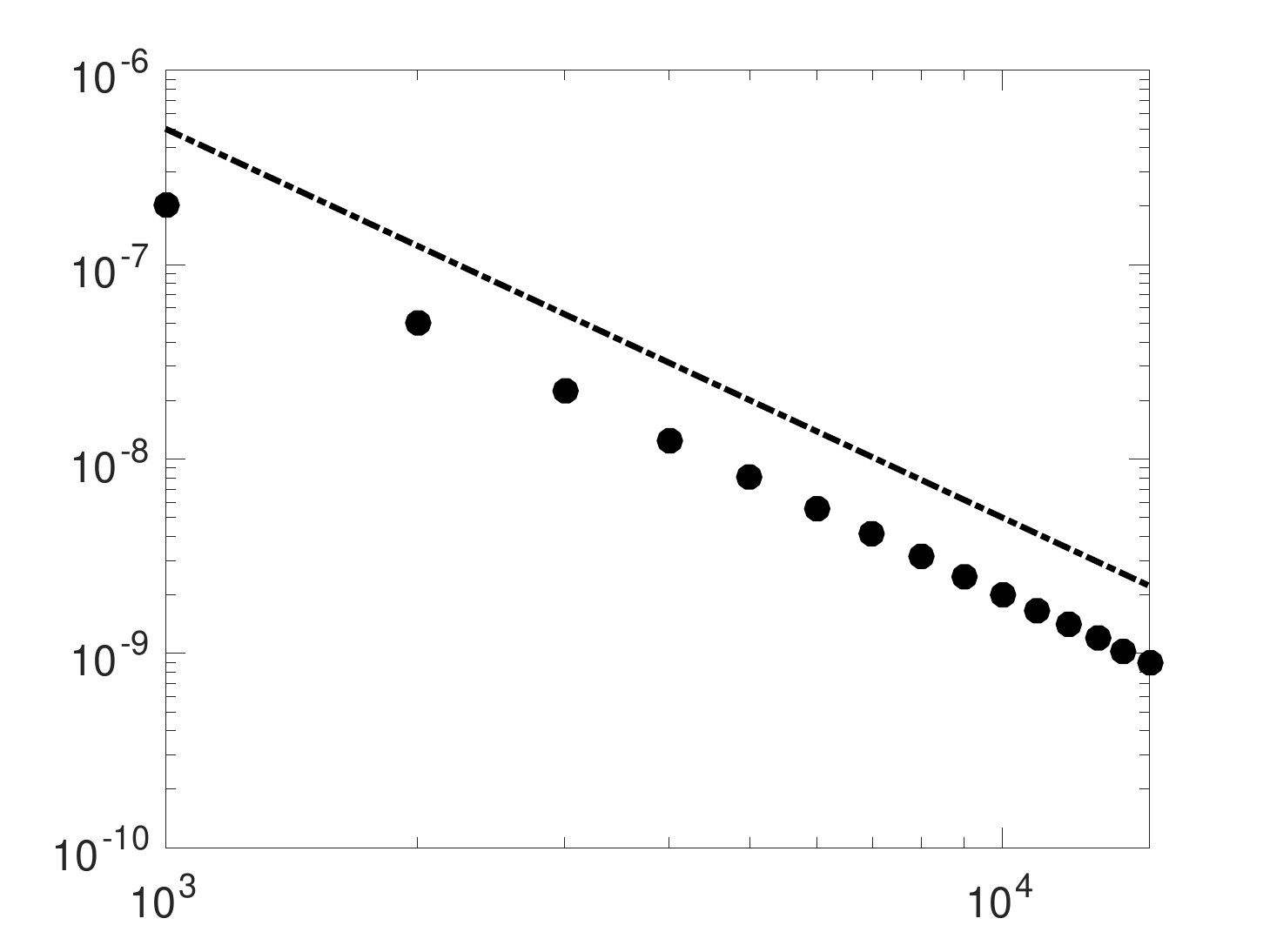}
 	\put (50,-2) {$\displaystyle n$}
 	\put(50,50) {\rotatebox{-22} {$\displaystyle\mathcal{O}(n^{-2})$}}
 	\put(35,72) {$\lvert\hat\lambda_n-\lambda^{asy}_n\rvert/\lambda^{asy}_n$}
 	\end{overpic}
  \end{minipage}
  \caption{Left: The large eigenvalues of~\cref{eqn:regular SLEP} are computed by \texttt{contFEAST} (see Algorithm~\ref{alg:contFEAST}) using search regions given by asymptotic estimates for the eigenvalues~\cref{eqn:SLEP asymptotics}. Right: The relative difference $\lvert\hat\lambda_n-\lambda^{asy}_n\rvert/\lambda^{asy}_n$ between the eigenvalues $\hat\lambda_n$ computed by \texttt{contFEAST} and the asymptotic values $\lambda^{asy}_n$ from~\cref{eqn:SLEP asymptotics}. The difference is compared to an $\mathcal{O}(n^{-2})$ relative error estimate~\cite{akbarfam2004higher}.}
\end{figure}
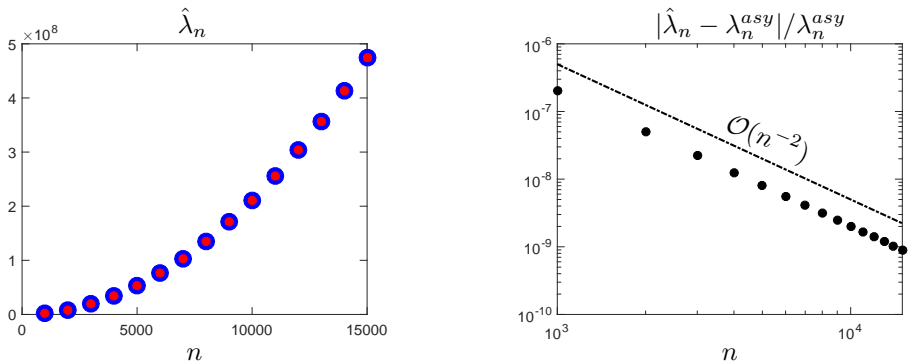

\subsubsection{An indefinite Sturm--Liouville eigenvalue problem}
\label{subsec:singular SLEP}
Next, we consider the following indefinite SLEP:
\begin{equation}
\label{eqn:singularly perturbed SLEP}
-\frac{d^2u}{dx^2}=\lambda x^3u, \qquad u(\pm 1)=0,
\end{equation}
which is closely related to models of light propagation in a nonhomogeneous material~\cite{akbarfam2004higher,wasow1986linear}. Since the weight function $x^3$ changes sign at $x=0$,~\cref{eqn:singularly perturbed SLEP} has a bi-infinite sequence of eigenvalues~\cite{atkinson1987asymptotics}. We index them in order as $\cdots\leq \lambda_{-2}\leq \lambda_{-1}< 0 <\lambda_1\leq \lambda_2\leq\cdots$.  The asymptotics for the positive eigenvalues are given by~\cite{akbarfam2002higher}
\begin{equation}
\label{eqn:singularly perturbed SLEP asymptotics}
\sqrt{\lambda_n}\sim \frac{(n-1/4)\pi}{\int_0^1x^{3/2}\, dx}, \qquad \lambda_n>0, \qquad n\rightarrow\infty.
\end{equation}
A similar expansion holds for the negative eigenvalues~\cite{akbarfam2002higher}.

In contrast to the previous example, the indefinite weight function $x^3$ means that~\cref{eqn:singularly perturbed SLEP} is not immediately associated with a self-adjoint operator on a Hilbert space. Instead,~\cref{eqn:singularly perturbed SLEP} is usually studied through the lens of a Krein space and the eigenfunctions form a Riesz basis for the Hilbert space with the inner product~\cite{curgus2013riesz}
\begin{equation}
(v,u)_{\lvert w\rvert}=\int_{-1}^1 \overline{v}u\lvert x\rvert^3\, dx.
\end{equation}

We use the leading order asymptotics in~\cref{eqn:singularly perturbed SLEP asymptotics} to identify search regions that are likely to contain an eigenvalue of~\cref{eqn:singularly perturbed SLEP}. Because the ultraspherical spectral method used to solve the ODEs in step 2 of Algorithm~\ref{alg:contFEAST} is efficient when applied to ODEs with smooth variable coefficients, it is convenient to treat~\cref{eqn:singularly perturbed SLEP} as a generalized eigenvalue problem, i.e., as $\mathcal{L}_1u=\lambda\mathcal{L}_2u$, where $\mathcal{L}_1u=-\frac{d^2u}{dx^2}$ and $\mathcal{L}_2u=x^3 u$. The eigenvalues of the pencil $z\mathcal{L}_2-\mathcal{L}_1$ are then computed with a straightforward generalization of Algorithm~\ref{alg:contFEAST} that is based on the spectral projector for the generalized eigenvalue problem, i.e.,
\begin{equation}
\label{eqn:generalized spectral projector}
\mathcal{P}_\mathcal{V}=\frac{1}{2\pi i}\int_{\partial\Omega}(z\mathcal{L}_2-\mathcal{L}_1)^{-1}\mathcal{L}_2\, dz.
\end{equation}
The eigenvalues and eigenfunctions are automatically resolved to essentially machine precision because of the use of the adaptive QR solver (see~\cref{fig:singular SLEP}).
\begin{figure}[!tbp]
\label{fig:singular SLEP}
  \centering
  \begin{minipage}[b]{0.45\textwidth}
    \begin{overpic}[width=\textwidth]{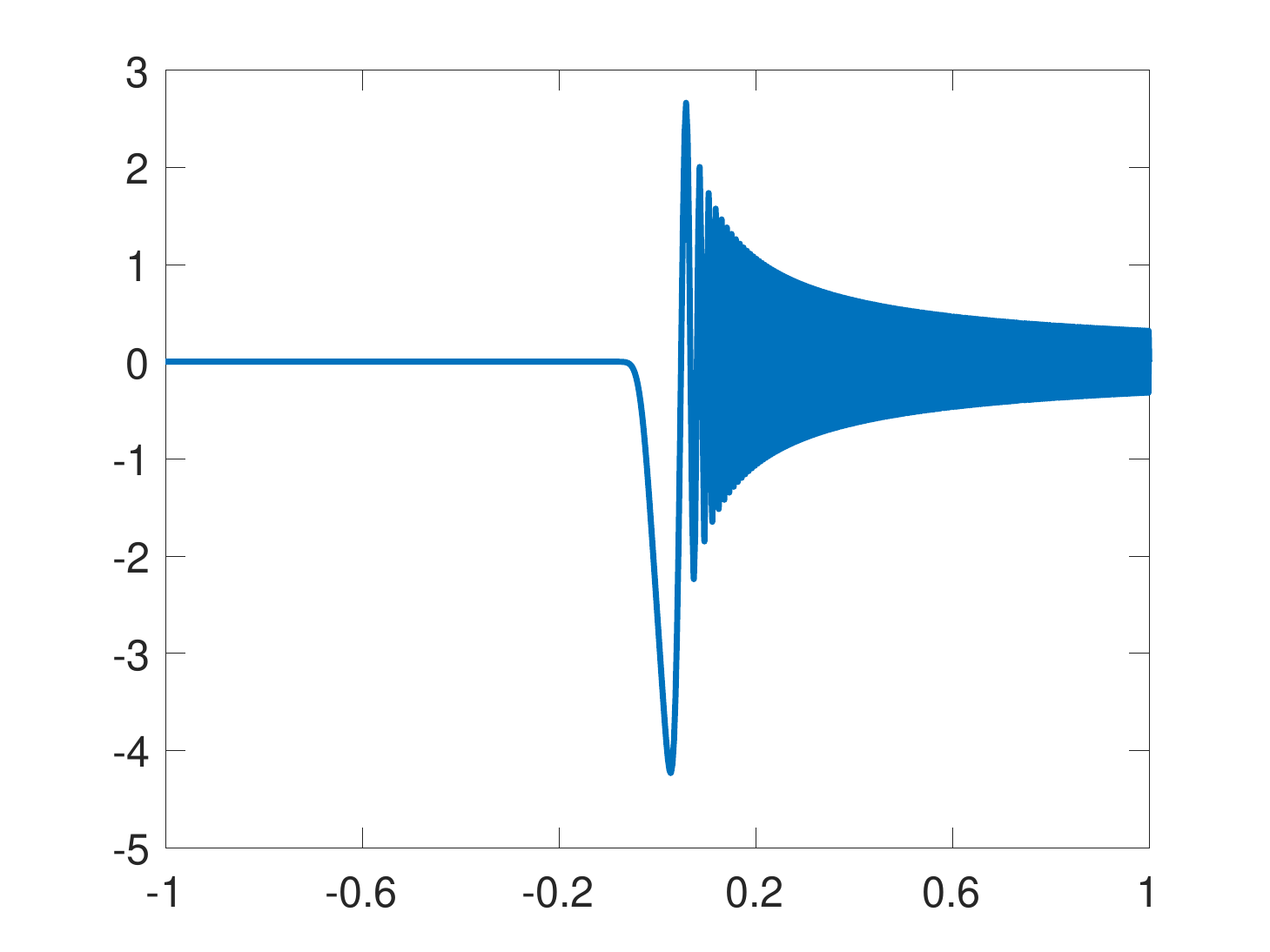}
 	\put (50,-2) {$\displaystyle x$}
 	\put(46,72) {$u_{1500}$}
	\end{overpic}
  \end{minipage}
  \hfill
  \begin{minipage}[b]{0.45\textwidth}
    \begin{overpic}[width=\textwidth]{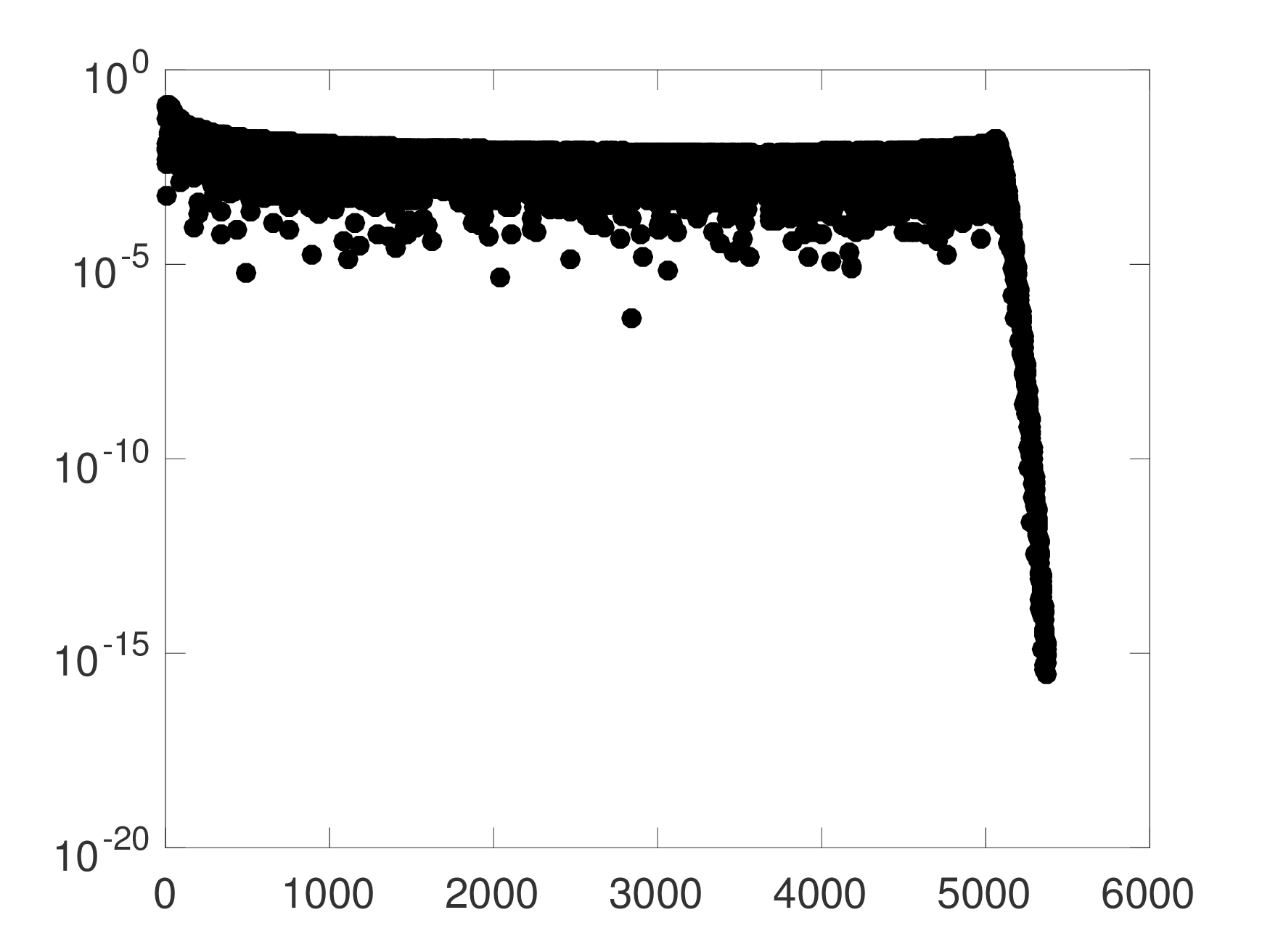}
 	\put (50,-2) {$\displaystyle k$}
 	\put(47.2,72) {$\lvert\hat u_k\rvert$}
	\end{overpic}
  \end{minipage}
  \caption{Left: The high-frequency eigenfunction associated to $\lambda_{1500}$ of the indefinite Sturm--Liouville eigenvalue problem~\cref{eqn:singularly perturbed SLEP} computed by \texttt{contFEAST} (see Algorithm~\ref{alg:contFEAST}). Right: The Chebyshev coefficients $\{\hat u_k\}$ in a series expansion used to represent the eigenfunction. About $5371$ Chebyshev coefficients are needed to accurately resolve the eigenfunction. The rapid decay in the coefficients to essentially machine precision is a good indication that the solution is fully resolved.}
\end{figure}

\section{Convergence and stability}\label{sec:convergence and stability}
The primary consequence of the approximations introduced in Algorithm~\ref{alg:contFEAST} is that the spectral projector is no longer applied exactly. Therefore, the basis $\hat Q$ computed for the Rayleigh--Ritz projection is not an exact basis for the invariant subspace $\mathcal{V}$ of $\mathcal{L}$ and may require further refinement. Here, we view the iterative refinement procedure used in Algorithm~\ref{alg:contFEAST} as a rational subspace iteration applied to a normal differential operator $\mathcal{L}$ in order to provide a preliminary analysis of the stability of the iteration and the sensitivity of the Ritz values. The main results may be summarized as follows.
\begin{itemize}[leftmargin=*,noitemsep]
\item[(i)] Algorithm~\ref{alg:contFEAST} yields a sequence of quasimatrices~$\hat Q_1,\dots,\hat Q_k$ that (generically) converge geometrically to an orthonormal basis for the eigenspace $\mathcal{V}$ (see~\cref{thm:rational subspace iteration for diff ops}).
\item[(ii)] If $\hat Q_k$ is a sufficiently good approximation to an orthonormal basis for $\mathcal{V}$, then the $\epsilon$-psuedospectrum of $\hat Q_k^* \mathcal{L}\hat Q_k$ is contained in the $2\epsilon$-psuedospectrum of $\mathcal{L}$ itself (see~\cref{thm:pseudospectra inclusion 2}).
\item[(iii)] Under mild conditions on the initial quasimatrix $F$ in Algorithm~\ref{alg:contFEAST}, the sequence $\lVert\mathcal{L}(\hat Q_k-Q)\rVert_{\mathbb{C}^m\rightarrow\mathcal{H}}$ is uniformly bounded as $k\rightarrow\infty$ (see~\cref{lem:L is UB on Q_k}).
\end{itemize}
Taken together, these results demonstrate that each iteration of Algorithm~\ref{alg:contFEAST} yields uniformly consistent Ritz pairs that converge linearly to the desired eigenpair and that the unboundedness of $\mathcal{L}$ does not lead to instability. Note that this analysis does not take into account the impact of finite-precision arithmetic or the fact that the shifted differential equations in~\cref{eqn:shifted linear ODEs} are not solved exactly at each iteration (see the discussion at the end of~\cref{subsec:rational subspace iteration for diff ops}). However, (ii) ensures that the eigenvalues of the small matrix $\hat Q_k^*\mathcal{L}\hat Q_k$ are not much more sensitive than the eigenvalues of $\mathcal{L}$. Therefore, provided that the eigenvalue problem for $\mathcal{L}$ is well-conditioned and we compute a sufficiently accurate approximation to a basis for $\mathcal{V}$, then we expect that the eigenvalues computed with Algorithm~\ref{alg:contFEAST} provide an accurate approximation to the desired eigenvalues of $\mathcal{L}$.

\subsection{Rational subspace iteration for differential operators}\label{subsec:rational subspace iteration for diff ops} 

In analogy to the matrix case~\cite{SubspaceIter}, Algorithm~\ref{alg:contFEAST} may be interpreted as a filtered subspace iteration. Filtered subspace iteration is a variant of standard subspace iteration for computing a target subset of eigenvalues of a matrix $A$~\cite[Ch.~5]{saad1992numerical}. The main idea is to choose a filter function $s(\cdot)$ that is large on the targeted eigenvalues of $A$ and small on the unwanted eigenvalues of $A$. Applying the spectral transformation $s(A)$,\footnote{A spectral transformation $s(\cdot)$ may be applied to $A$ via the eigendecomposition of $A$, or more generally the Jordan decomposition. For example, if $A$ has eigendecomposition $A=X\Lambda X^{-1}$ with $\Lambda={\rm diag}(\lambda_1,\dots,\lambda_n)$, then $s(A)=X\,s(\Lambda)X^{-1}$, where $s(\Lambda)={\rm diag}(s(\lambda_1),\dots,s(\lambda_n))$.} one uses standard subspace iteration to compute a basis for the eigenspace of $s(A)$ corresponding to its largest eigenvalues, i.e., the targeted eigenvalues of $A$. With an approximate basis for the eigenspace available, the eigenvalues and eigenvectors can be extracted with a Rayleigh--Ritz step.

From this perspective, Algorithm~\ref{alg:contFEAST} computes the eigenvalues of $\mathcal{L}$ in $\Omega$ with the aid of a rational filter function induced by the quadrature rule in~\cref{eqn:approximate spectral projector}, i.e., $s(\cdot)$ is given by
\begin{equation}\label{eqn:rational filter}
s(z)=\sum_{k=1}^\ell\frac{w_k}{z_k-z}, \qquad z\in\mathbb{C}\setminus\{z_1,\dots,z_l\}.
\end{equation}
The functional calculus for unbounded normal operators ensures that if $\lambda_i$ is an eigenvalue of $\mathcal{L}$ with eigenfunction $u_i$, then $s(\lambda_i)$ is an eigenvalue of $s(\mathcal{L})$ with eigenfunction $u_i$~\cite[VIII.5]{reed1980methods}.\footnote{The result~\cite[VIII.5]{reed1980methods} is stated for closed self-adjoint operators on $\mathcal{H}$, however, it extends immediately to closed normal operators on $\mathcal{H}$ if the spectral decomposition of a closed normal operator~\cite[Theorem 13.33]{rudin1991functional} is used in place of the spectral decomposition of a self-adjoint operator. For information on the spectral decomposition of unbounded normal operators and the associated functional calculus, see~\cite[Ch. 13]{rudin1991functional} and~\cite{dunford1958survey}.} As the degree of the quadrature rule is increased, the rational function becomes an increasingly good approximation to the Cauchy integral
\begin{equation}\label{eqn:Cauchy_int}
\chi(z)=\frac{1}{2\pi i}\int_{\partial\Omega}\frac{dw}{w-z}, \qquad z\in\mathbb{C}\setminus\partial\Omega.
\end{equation}
Therefore, the eigenvalues of $\mathcal{L}$ in $\Omega$ are usually $\mathcal{O}(1)$ in size under the spectral transformation $s(\cdot)$ while the eigenvalues outside of $\Omega$ are much smaller.

We now turn to the convergence of the iteration described in Algorithm~\ref{alg:contFEAST}, which we interpret as a subspace iteration applied to the bounded linear operator $\mathcal{\hat P}_\mathcal{V}=s(\mathcal{L})$. It is helpful to introduce the notions of the \textit{spectral radius} and a \textit{dominant eigenspace} of a bounded linear operator $\mathcal{B}$. The spectral radius of a bounded linear operator $\mathcal{B}$ on a Hilbert space $\mathcal{H}$ is defined as~\cite[p.~99]{davies2007linear}
\begin{equation}
\label{eqn:spectral radius}
\rho(\mathcal{B})=\max \{\lvert z\rvert : z\in\lambda(\mathcal{B})\}.
\end{equation}
The spectral radius is useful because it characterizes the asymptotic behavior of $\lVert\mathcal{B}^k\rVert_\mathcal{H}$, in the sense that $\rho(\mathcal{B})=\lim_{k\rightarrow\infty}\lVert\mathcal{B}^k\rVert_\mathcal{H}^{1/k}$~\cite[Theorem 4.1.3]{davies2007linear}. Let $\mathcal{V}$ be an invariant subspace of $\mathcal{B}$ associated with eigenvalues $\lambda_1\geq\dots\geq\lambda_m$ and a spectral projector $\mathcal{P}_\mathcal{V}$. We say that $\mathcal{B}$ has dominant eigenspace $\mathcal{V}$ if
\begin{equation}
\label{eqn:dominant eigenspace}
\rho(\mathcal{(I-P}_\mathcal{V})\mathcal{B})<|\lambda_m|.
\end{equation}

The following theorem is an extension of a convergence analysis~\cite[p.~119]{saad1992numerical} for matrix subspace iteration to the setting of bounded linear operators with a dominant eigenspace. We omit the details of the proof, as they are identical to those found in the proof of Lemma 3.1 and Lemma 3.2 of~\cite{gopalakrishnan2017filtered}.
\begin{theorem}
\label{thm:rational subspace iteration for diff ops}
Let $\mathcal{B}$ be a bounded linear operator on a Hilbert space $\mathcal{H}$ with dominant eigenspace $\mathcal{V}$, defined in~\cref{eqn:dominant eigenspace}, having ${\rm dim}(\mathcal{V})=m$. Select a quasimatrix $F:\mathbb{C}^m\rightarrow\mathcal{H}$ such that the columns of $\mathcal{P}_\mathcal{V}F$ are linearly independent and suppose the columns of the quasimatrix $\hat Q_k:\mathbb{C}^m\rightarrow\mathcal{H}$ form an orthonormal basis for ${\rm range}(\mathcal{B}^kF)$, for $k=1,2,3,\dots$. If $u\in\mathcal{V}$ is an eigenvector of $\mathcal{B}$ with eigenvalue $\lambda$, then there is a function $\hat u_k\in{\rm range}(\hat Q_k)$ such that \[
\lVert\hat u_k-u\rVert_\mathcal{H}\leq \left(\left|\rho/\lambda\right|+\epsilon_k\right)^k\lVert (I-\mathcal{P}_\mathcal{V})Fx\rVert_\mathcal{H},\qquad  k=1,2,3,\dots,
\] 
where $\rho=\rho((\mathcal{I-P}_\mathcal{V})\mathcal{B})$, $\epsilon_k\rightarrow 0$ as $k\rightarrow\infty$, and $u=\mathcal{P}_\mathcal{V}Fx$.
\end{theorem}

Although we have neglected the effects of approximately solving the ODEs in~\cref{eqn:shifted linear ODEs} and the impact of round-off errors in our brief analysis of rational subspace iteration for normal differential operators, we mention two recent results for rational subspace iteration with matrices~\cite{saad2016analysis} and self-adjoint differential operators~\cite{gopalakrishnan2017filtered,gopalakrishnan2017spectral}.
\begin{itemize}[leftmargin=*,noitemsep]
\item[•] For matrices, small errors made during application of the spectral projector generally do not alter the convergence behavior of subspace iteration~\cite{saad2016analysis}. In this case, the sequence $\hat Q_k$ no longer converges to an exact basis for $\mathcal{V}$. However, the matrices $\hat Q_k$ approximate a basis for $\mathcal{V}$ and the approximation error converges geometrically to a constant determined by the sizes of the errors introduced at each iteration~\cite{saad2016analysis}. 
\item[•] For self-adjoint differential operators (closed and densely defined on $\mathcal{H}$), rational subspace iteration converges to a subspace even when the resolvent operator is discretized to solve the ODEs in~\cref{eqn:shifted linear ODEs}~\cite{gopalakrishnan2017filtered,gopalakrishnan2017spectral}. The distance between the computed subspace and the target eigenspace (in a distance metric between subspaces) is proportional to the approximation error in the discretized resolvent~\cite{gopalakrishnan2017filtered,gopalakrishnan2017spectral}.
\end{itemize}
We expect that similar statements hold for normal operators on $\mathcal{H}$, but a rigorous and detailed convergence analysis is more subtle and beyond the scope of this paper.

\subsection{A pseudospectral inclusion theorem}\label{subsec:pseudospectral inclusion}
As ${\rm range}(\hat Q)$ is not an invariant subspace of $\mathcal{L}$, the $\epsilon$-pseudospectrum of $\hat Q^*\mathcal{L}\hat Q$ is not, in general, contained in the $\epsilon$-pseudospectrum of $\mathcal{L}$. 
However, if $\lVert\hat Q-Q\rVert_{\mathbb{C}^m\rightarrow\mathcal{H}}$ is sufficiently small, then the $\epsilon$-pseudospectrum of $\hat Q^*\mathcal{L}\hat Q$ is contained in the $2\epsilon$-pseudospectrum of $\mathcal{L}$. 

\begin{theorem}\label{thm:pseudospectra inclusion 2}
Consider a closed operator $\mathcal{L}$ with domain $\mathcal{D}(\mathcal{L})$ that is densely defined on a Hilbert space $\mathcal{H}$ and fix $\epsilon>0$. Let $Q:\mathbb{C}^m\rightarrow\mathcal{D}(\mathcal{L})\bigcap\mathcal{D}(\mathcal{L}^*)$ satisfy $Q^*Q=I$, and let ${\rm range}(Q)$ be an $m$-dimensional invariant subspace of $\mathcal{L}$. If $\hat Q:\mathbb{C}^m\rightarrow\mathcal{D}(\mathcal{L})$ satisfies
\[
\lVert\hat Q-Q\rVert_{\mathbb{C}^m\rightarrow\mathcal{H}}\left(\lVert\mathcal{L}^*Q\rVert_{\mathbb{C}^m\rightarrow\mathcal{H}}+\lVert\mathcal{L}Q\rVert_{\mathbb{C}^m\rightarrow\mathcal{H}}+\lVert\mathcal{L}(\hat Q-Q)\rVert_{\mathbb{C}^m\rightarrow\mathcal{H}}\right)<\frac{\epsilon}{2},
\]
then $\lambda_\epsilon(\hat Q^*\mathcal{L}\hat Q)\subset\lambda_{2\epsilon}(\mathcal{L})$.
\end{theorem}
\begin{proof}

Consider $z\in\lambda_\epsilon(\hat Q^*\mathcal{L}\hat Q)$. If $z\in\lambda_\epsilon(\hat Q^*\mathcal{L}\hat Q)\cap\lambda_\epsilon(\mathcal{L})$, there is nothing to prove, so assume without loss of generality that $z\not\in\lambda_\epsilon(\mathcal{L})$. If we denote $R_Q(z)=(zI-Q^*\mathcal{L}Q)^{-1}$, $R_{\hat Q}(z)=(zI-\hat Q^*\mathcal{L}\hat Q)^{-1}$, and $E=\hat Q-Q$, then we have that $R_{\hat Q}(z)=[R_Q(z)^{-1}-B]^{-1}$, where $B=Q^*\mathcal{L}E+E^*\mathcal{L}Q+E^*\mathcal{L}E$. Employing a formula for the inverse of the sum of two matrices, we obtain $R_{\hat Q}(z)=R_Q(z)+R_Q(z)[I-BR_Q(z)]^{-1}BR_Q(z)$~\cite{henderson1981deriving}.

Now, $\lVert B\rVert_{\mathbb{C}^m}\leq\lVert Q^*\mathcal{L}E\rVert_{\mathbb{C}^m}+\lVert E^*\mathcal{L}Q\rVert_{\mathbb{C}^m}+\lVert E^*\mathcal{L}E\rVert_{\mathbb{C}^m}$. Since $\lVert E^*\rVert_{\mathcal{H}\rightarrow\mathbb{C}^m}=\lVert E\rVert_{\mathbb{C}^m\rightarrow\mathcal{H}}$ and $\lVert Q^*\mathcal{L}\rVert_{\mathcal{H}\rightarrow\mathbb{C}^m}=\lVert\mathcal{L}^*Q\rVert_{\mathbb{C}^m\rightarrow\mathcal{H}}$~\cite[p.~256]{kato2013perturbation}, our hypothesis indicates that the sum of the three terms comprising $B$ is bounded in norm by
\[
\lVert B\rVert_{\mathbb{C}^m}\leq\lVert E\rVert_{\mathbb{C}^m\rightarrow\mathcal{H}}\left(\lVert\mathcal{L}^*Q\rVert_{\mathbb{C}^m\rightarrow\mathcal{H}}+\lVert\mathcal{L}Q\rVert_{\mathbb{C}^m\rightarrow\mathcal{H}}+\lVert\mathcal{L}E\rVert_{\mathbb{C}^m\rightarrow\mathcal{H}}\right)< \frac{\epsilon}{2}.
\]
Moreover, since $z\not\in\lambda_\epsilon(\mathcal{L})$, we have that $\lVert R_Q(z)\rVert_{\mathbb{C}^m}\leq 1/\epsilon$ by~\cref{thm:pseudospectra inclusion 1}. Therefore, $\lVert BR_Q(z)\rVert_{\mathbb{C}^m}\leq 1/2$.

Because $\lVert BR_Q(z)\rVert_{\mathbb{C}^m}\leq 1/2$, we may use the Neumann series to compute $(I-BR_Q(z))^{-1}=\sum_{k=0}^\infty (BR_Q(z))^k$. We see that $R_{\hat Q}(z)=R_Q(z)\left(I+\sum_{k=1}^\infty(BR_Q(z))^k\right)$ and therefore, 
\[
\lVert R_{\hat Q}(z)\rVert_{\mathbb{C}^m}\leq \left(1+\sum_{k=1}^\infty \frac{1}{2^k}\right)\!\lVert R_Q(z)\rVert_{\mathbb{C}^m}= 2\lVert R_Q(z)\rVert_{\mathbb{C}^m}.
\]
Now, if $z\in\lambda_\epsilon(\hat Q^*\mathcal{L}\hat Q)$, then  $\lVert R_Q(z)\rVert_{\mathbb{C}^m}\geq\lVert R_{\hat Q}(z)\rVert_{\mathbb{C}^m}/2>1/(2\epsilon)$. By~\cref{thm:pseudospectra inclusion 1}, we have that $\lVert(zI-\mathcal{L})^{-1}\rVert_\mathcal{H}\geq\lVert R_Q(z)\rVert_{\mathbb{C}^m}$. Collecting inequalities yields the result $\lVert(zI-\mathcal{L})^{-1}\rVert_\mathcal{H}>1/(2\epsilon)$, i.e., $z\in\lambda_{2\epsilon}(\mathcal{L})$.
\end{proof}

A consequence of~\cref{thm:pseudospectra inclusion 2} is that Algorithm~\ref{alg:contFEAST} possesses a type of stability provided that $\mathcal{L}$ is uniformly bounded on the sequence $E_1,E_2,E_3,\dots$, where $E_k=\hat Q_k-Q$ for $k\geq 1$. If $\mathcal{L}$ is uniformly bounded on $\{E_k\}_{k=1}^\infty$, then there is a $\Lambda\geq 0$ such that ${\rm sup}_{k\geq 1}\lVert\mathcal{L}E_k\rVert_{\mathbb{C}^m\rightarrow\mathcal{H}}\leq\Lambda$. Applying~\cref{thm:pseudospectra inclusion 2}, we see that Algorithm~\ref{alg:contFEAST} computes elements in the $2\epsilon$-pseudospectrum of $\mathcal{L}$ provided that a basis for $\mathcal{V}$ is resolved to within $\epsilon/(2(\lVert\mathcal{L}^*Q\rVert_{\mathbb{C}^m\rightarrow\mathcal{H}}+\lVert\mathcal{L}Q\rVert_{\mathbb{C}^m\rightarrow\mathcal{H}}+\Lambda))$. 

We now verify, with two mild constraints placed on the choice of the initial quasimatrix $F$, that $\mathcal{L}$ is uniformly bounded on the sequence $\{\hat Q_k\}_{k=1}^\infty$ generated by Algorithm~\ref{alg:contFEAST}. Note that this implies that $\mathcal{L}$ is uniformly bounded on $\{E_k\}_{k=1}^\infty$ because $E_k=\hat Q_k-Q$ and ${\rm range}(Q)\subset\mathcal{D}(\mathcal{L})$. The constraints on $F$ are generically satisfied when $F$ is selected as in~\cref{sec:practical algorithm}. In the statement of the bound on $\lVert\mathcal{L}\hat Q_k\rVert_{\mathbb{C}^m\rightarrow\mathcal{H}}$, we use the notation $\sigma_{\min}(\mathcal{P}_\mathcal{V}F)$ and $\sigma_{\min}((\mathcal{I-P}_\mathcal{V})F)$ to denote the smallest singular values of the quasimatrices $\mathcal{P}_\mathcal{V}F$ and $(\mathcal{I-P}_\mathcal{V})F$, respectively.\footnote{The singular value decomposition of a quasimatrix $A:\mathbb{C}^m\rightarrow\mathcal{H}$ is the decomposition $A=U\Sigma V^*$, where $U:\mathbb{C}^m\rightarrow\mathcal{H}$ is a quasimatrix with $\mathcal{H}$-orthonormal columns, $\Sigma\in\mathbb{C}^{m\times m}$ is a diagonal matrix with non-negative entries $\sigma_1\geq\dots\geq\sigma_m$, and $V\in\mathbb{C}^{m\times m}$ is a unitary matrix~\cite{contFact2015}.}

\begin{lemma}
\label{lem:L is UB on Q_k}
Consider a closed, normal operator $\mathcal{L}$ with domain $\mathcal{D}(\mathcal{L})$ that is densely defined on a Hilbert space $\mathcal{H}$. Let $\mathcal{\hat P}_\mathcal{V}$ be the bounded operator on $\mathcal{H}$ defined in~\cref{eqn:approximate spectral projector} and suppose that $\mathcal{\hat P}_\mathcal{V}$ has a dominant eigenspace of $\mathcal{V}$ (see~\cref{eqn:dominant eigenspace}) with ${\rm dim}(\mathcal{V})=m$.
Let $F$, $\mathcal{P}_\mathcal{V}$, and $\{\hat Q_k\}_{k=1}^\infty$ be as in~\cref{thm:rational subspace iteration for diff ops} with $\mathcal{B}=\mathcal{\hat P}_\mathcal{V}$. Suppose that $\mathcal{\hat P}_\mathcal{V}^kF$ (for each $k\geq 1$) and $(\mathcal{I-P}_\mathcal{V})F$ each have linearly independent columns and that ${\rm range}(F)\subset\mathcal{D}(\mathcal{L})$.  Then, we have that 
\[
\lVert\mathcal{L}\hat Q_k\rVert_{\mathbb{C}^m\rightarrow\mathcal{H}}\leq 2M\lVert\mathcal{L}F\rVert_{\mathbb{C}^m\rightarrow\mathcal{H}},\qquad k=1,2,3\dots,
\]
where $M=\max\left\lbrace1/\sigma_{\min}(\mathcal{P}_\mathcal{V}F),1/\sigma_{\min}((\mathcal{I-P}_\mathcal{V})F)\right\rbrace$.
\end{lemma}
\begin{proof}
Since $\hat Q_k$ is an orthonormal basis for $\mathcal{\hat P}_\mathcal{V}^kF$, there is a matrix $R_k\in\mathbb{C}^{m\times m}$ such that $\mathcal{\hat P}_\mathcal{V}^kF=\hat Q_kR_k$. By the assumption that $\mathcal{\hat P}_\mathcal{V}^kF$ has linearly independent columns, we know that $R_k$ is invertible. We obtain that 
\begin{equation}
\label{eqn:kth basis}
\hat Q_k=\mathcal{\hat P}_\mathcal{V}^kFR_k^{-1}.
\end{equation} 
We use the spectral projector $\mathcal{P}_\mathcal{V}$ to rewrite~\cref{eqn:kth basis} as
\begin{equation}
\label{eqn:kth basis decomposed}
\hat Q_k=\mathcal{\hat P}_\mathcal{V}^k\left(\mathcal{P}_\mathcal{V}F+(\mathcal{I-P}_\mathcal{V})F\right)R_k^{-1}.
\end{equation}

Now, ${\rm range}(\mathcal{P}_\mathcal{V}F)$ and ${\rm range}((\mathcal{I-P}_\mathcal{V})F)$ are invariant under $\mathcal{\hat P}_\mathcal{V}$~\cite[p.~178]{kato2013perturbation}. Consequently, there are matrices $D_1,D_2\in\mathbb{C}^{m\times m}$ such that
\begin{equation}
\label{eqn:invariant powers}
\mathcal{\hat P}_\mathcal{V}^k\mathcal{P}_\mathcal{V}F=\mathcal{P}_\mathcal{V}FD_1^k, \qquad \mathcal{\hat P}_\mathcal{V}^k(\mathcal{I-P}_\mathcal{V})F=(\mathcal{I-P}_\mathcal{V})FD_2^k.
\end{equation}
Substituting~\cref{eqn:invariant powers} into~\cref{eqn:kth basis decomposed} yields the following useful equation for $\hat Q_k$:
\begin{equation}
\label{eqn:useful Q_k}
\hat Q_k=\left(\mathcal{P}_\mathcal{V}FD_1^k+(\mathcal{I-P}_\mathcal{V})FD_2^k\right)R_k^{-1}.
\end{equation}
Applying $\mathcal{L}$ to both sides of~\cref{eqn:useful Q_k} and commuting with the spectral projectors $\mathcal{P}_\mathcal{V}$ and $\mathcal{I-P}_\mathcal{V}$~\cite[p.~179]{kato2013perturbation}, we obtain
\begin{equation}
\label{eqn:useful LQ_k}
\mathcal{L}\hat Q_k=\left(\mathcal{P}_\mathcal{V}\mathcal{L}FD_1^k+(\mathcal{I-P}_\mathcal{V})\mathcal{L}FD_2^k\right)R_k^{-1}.
\end{equation}

Since ${\rm range}(F)\subset\mathcal{D}(\mathcal{L})$, we have that $\lVert\mathcal{L}F\rVert_{\mathbb{C}^m\rightarrow\mathcal{H}}<\infty$. Additionally, since $\mathcal{L}$ is normal, the spectral projectors have norms equal to $1$~\cite[p.~277]{kato2013perturbation}. Therefore, it remains to find a uniform bound for $\lVert D_1^kR_k^{-1}\rVert_{\mathbb{C}^m}$ and $\lVert D_2^kR_k^{-1}\rVert_{\mathbb{C}^m}$ as $k\rightarrow\infty$. 

For brevity, we prove uniform boundedness of $\lVert D_1^kR_k^{-1}\rVert_{\mathbb{C}^m}$ and note that the proof for $\lVert D_2^kR_k^{-1}\rVert_{\mathbb{C}^m}$ is essentially identical. We begin by commuting $\mathcal{\hat P}_\mathcal{V}$ with the spectral projectors in~\cref{eqn:invariant powers} and substituting the QR factorization of $\mathcal{\hat P}_\mathcal{V}^kF$ to see that
\begin{equation}
\label{eqn:convergence}
\mathcal{P}_\mathcal{V}\hat Q_k=(\mathcal{P}_\mathcal{V}F)D_1^kR_k^{-1}.
\end{equation}
Using the psuedoinverse $(\mathcal{P}_\mathcal{V}F)^+$ of the quasimatrix\footnote{The pseudoinverse of a quasimatrix $A:\mathbb{C}^m\rightarrow\mathcal{H}$ may be defined via the SVD as $A^+=V\Sigma^{+}U^*$, where $\Sigma^+$ is the diagonal matrix with entries $\Sigma^+_{ii}=1/\sigma_i$ if $\sigma_i\neq 0$ and $0$ otherwise. It is easy to verify familiar properties from the matrix case~\cite[p.~290]{golub2012matrix}, i.e., if $A$ has linearly independent columns, then $A^+A=I$ and $\lVert A^+\rVert_{\mathcal{H}\rightarrow\mathbb{C}^m}=1/\sigma_{\min}(A)$.} $\mathcal{P}_\mathcal{V}F$ and noting that $\mathcal{P}_\mathcal{V}F$ has linearly independent columns,~\cref{eqn:convergence} implies that
\begin{equation}
\label{eqn:D^kR_k}
D_1^kR_k^{-1}=(\mathcal{P}_\mathcal{V}F)^+\mathcal{P}_\mathcal{V}\hat Q_k.
\end{equation}

Now, we know that $\lVert\mathcal{P}_\mathcal{V}\hat Q_k\rVert_\mathcal{H}\leq 1$, because $\lVert\mathcal{P}_\mathcal{V}\rVert_\mathcal{H}=1$ and $\hat Q_k$ has orthonormal columns. We conclude that
\begin{equation}
\label{eqn:bounded D1^kR_k}
\lVert D_1^kR_k^{-1}\rVert_{\mathbb{C}^m}\leq\frac{1}{\sigma_{\min}(\mathcal{P}_\mathcal{V}F)}.
\end{equation}
A similar argument shows that
\begin{equation}
\label{eqn:bounded D2^kR_k}
\lVert D_2^kR_k^{-1}\rVert_{\mathbb{C}^m}\leq\frac{1}{\sigma_{\min}((\mathcal{I-P}_\mathcal{V})F)}.
\end{equation}
Taking norms in~\cref{eqn:useful LQ_k} and substituting the bounds from~\cref{eqn:bounded D1^kR_k} and~\cref{eqn:bounded D2^kR_k}, we find
\begin{equation}
\label{eqn:full thm bound}
\lVert\mathcal{L}\hat Q_k\rVert_{\mathbb{C}^m\rightarrow\mathcal{H}}\leq\lVert\mathcal{L}F\rVert_{\mathbb{C}^m\rightarrow\mathcal{H}}\left(\frac{1}{\sigma_{\min}(\mathcal{P}_\mathcal{V}F)}+\frac{1}{\sigma_{\min}((\mathcal{I-P}_\mathcal{V})F)}\right).
\end{equation}
The lemma follows immediately from~\cref{eqn:full thm bound}.
\end{proof}

\Cref{thm:rational subspace iteration for diff ops},~\cref{thm:pseudospectra inclusion 2}, and~\cref{lem:L is UB on Q_k} provide a preliminary analysis to explain why Algorithm~\ref{alg:contFEAST} accurately computes the eigenvalues of normal operators with a dominant eigenspace $\mathcal{V}$. \Cref{thm:rational subspace iteration for diff ops} allows us to accurately resolve an orthonormal basis $Q$ for $\mathcal{V}$ by refining the quasimatrix $\hat Q_k$ with subspace iteration. \Cref{lem:L is UB on Q_k} confirms that $\mathcal{L}\hat Q_k$ does not grow without bound as $\hat Q_k$ is refined. Finally,~\cref{thm:pseudospectra inclusion 2} demonstrates that the eigenvalues are computed to the expected accuracy, provided that the basis for $\mathcal{V}$ has been resolved.

\section{An operator analogue of the Rayleigh Quotient Iteration}\label{sec:continuous RQI}
It is useful to have operator analogues for other eigensolvers too; particularly, when the eigenvalues of interest are difficult to target with a pre-selected search region $\Omega\subset \mathbb{C}$. The Rayleigh Quotient Iteration (RQI) is a generalization of the inverse iteration that incorporates dynamic shifting to obtain cubic (for Hermitian problems) or quadratic (non-Hermitian problems) convergence~\cite{parlett1974rayleigh}. Given a matrix $A\in\mathbb{C}^{n\times n}$ and an initial vector $\tilde y_0\in\mathbb{C}^n$, RQI computes the iterates 
\begin{equation}
\label{eqn: RQI}
\tilde y_{k+1}=(A-\beta_kI)^{-1}y_k, \qquad \beta_k=y_k^*Ay_k,\qquad y_k=\frac{\tilde y_k}{\lVert \tilde y_k\rVert_2},\qquad k=0,1,2,\dots.
\end{equation}
The vectors $y_k$ typically converge to a nearby eigenvector of $A$, while the sequence $\beta_k$ converges to the associated eigenvalue of $A$~\cite{ostrowski1957convergence}. In the matrix setting,~\cref{eqn: RQI} is often used to compute interior eigenvalues or refine an estimate of an invariant subspace~\cite{pantazis1995regions,parlett1974rayleigh}. 

Replacing a matrix $A$ by a differential operator $\mathcal{L}:\mathcal{D}(\mathcal{L})\rightarrow\mathcal{H}$, as in~\cref{eqn:differential eigenvalue problem}, and the vectors $\tilde y_k$ by functions $f_k\in\mathcal{D}(\mathcal{L})$, we obtain an operator analogue of RQI.
One needs to select an initial function $f_0\in\mathcal{D}(\mathcal{L})$ and solve a sequence of ODEs, i.e.,
\begin{equation}
\label{eqn: contRQI}
(\mathcal{L}-\beta_k I)f_{k+1}=f_k, \qquad f_{k+1}(\pm 1)=\dots=f_{k+1}^{(N/2)}(\pm 1)=0.
\end{equation}
At each iteration, the shift $\beta_k$ is computed from the Rayleigh Quotient $(f_k,\mathcal{L}f_k)_\mathcal{H}$ (in strong form) and the solution $f_{k+1}$ is normalized after each iteration. Analogous to the matrix setting, we observe that the operator analogue of the Rayleigh Quotient Iteration converges cubically for self-adjoint operators and quadratically otherwise~\cite{ChebfunExample}.

We note that block generalizations of RQI (RSQR and GRQI~\cite{absil2004cubically}) are also easily extended to the differential operator setting. In this case, a sequence of quasimatrices $\hat Q_k$ with $\mathcal{H}$-orthonormal columns are generated to approximate an invariant subspace of $\mathcal{L}$ and a Rayleigh--Ritz projection is performed to compute approximate eigenvalues and eigenvectors. As with the operator analogue of FEAST,~\cref{thm:pseudospectra inclusion 2} implies that the iteration~\cref{eqn: contRQI} accurately computes eigenvalues of normal differential operators when the basis for the target eigenspace is sufficiently resolved. 

\subsection{Free vibrations of an airplane wing}\label{subsubsec:beam eqn}
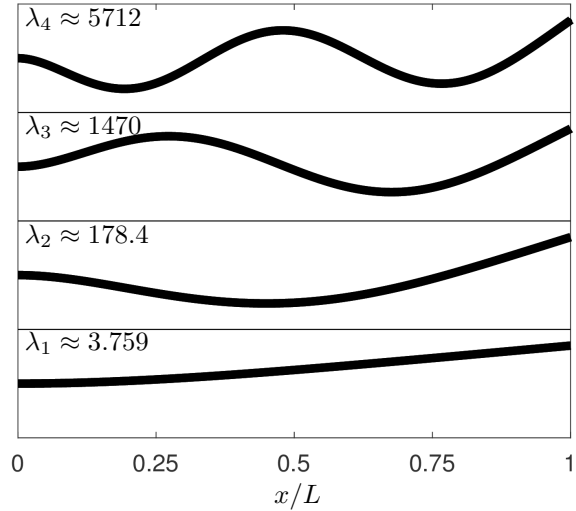
\begin{figure}[!tbp]
\label{fig:beam eqn}
  \centering
  \begin{minipage}[b]{0.72\textwidth}
    \begin{overpic}[width=\textwidth]{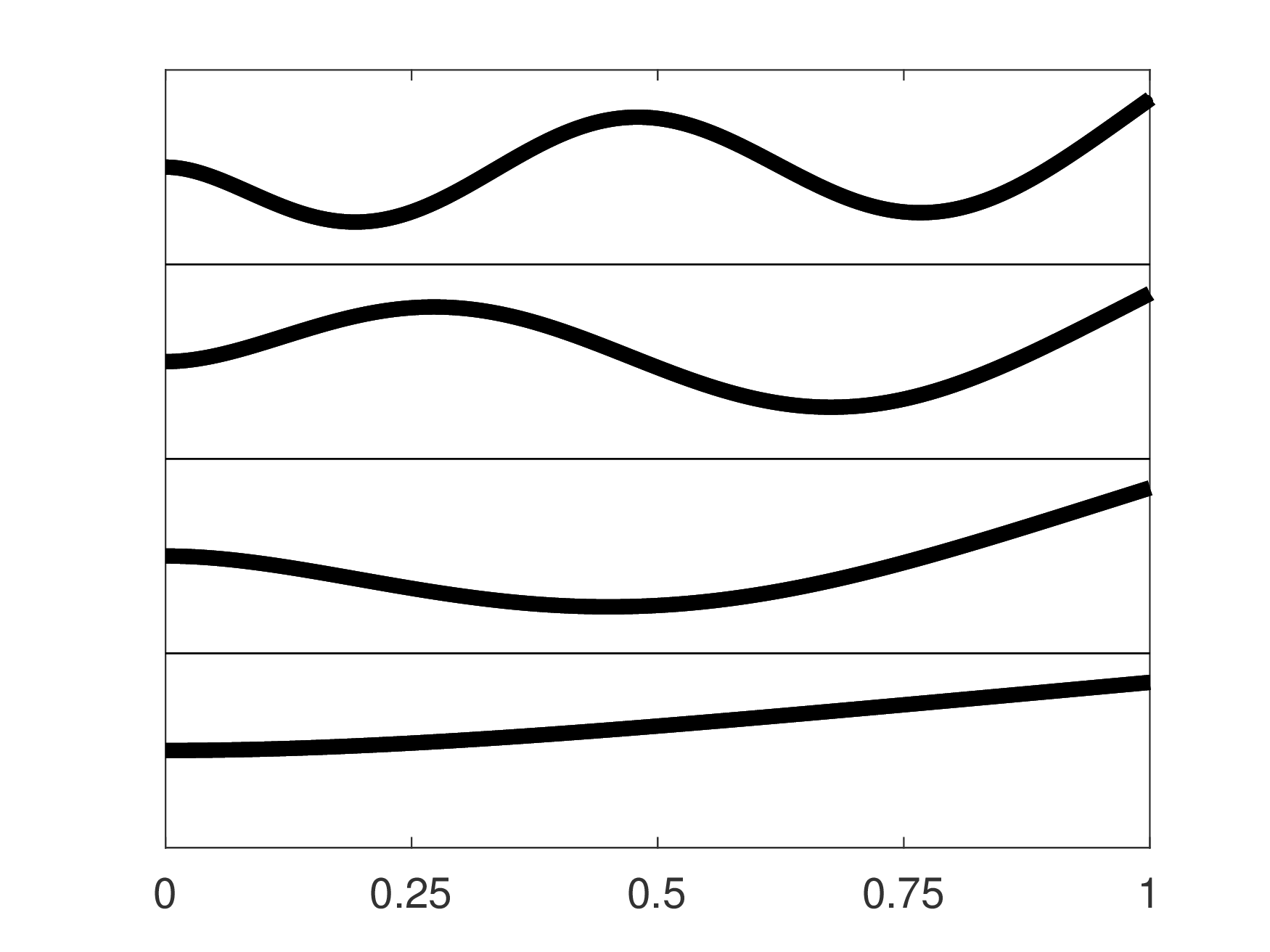}
 	\put (49,-1) {$x/L$}
 	\put (14,20.6) {$\lambda_1\approx 3.759$}
 	\put (14,35.72) {$\lambda_2\approx 178.4$}
 	\put (14,51.2) {$\lambda_3\approx 1470$}
 	\put (14,66.45) {$\lambda_4\approx 5712$}
	\end{overpic}
  \end{minipage}
  \caption{Selected free-vibration modes of an airplane wing modeled by~\cref{eqn:cantilever beam}.}
\end{figure}

The improved convergence rate of RQI can offer much faster computation time than subspace iteration, often requiring only $3$ or $4$ ODE solves to reach an accuracy of essentially machine precision~\cite{ChebfunExample}. We now employ~\cref{eqn: contRQI} for the rapid computation of vibrational modes of an airplane wing.

An airplane wing may be crudely modeled as a thin, cantilevered beam of length $L$ with a linear taper. The governing equation for free vibrations is~\cite{han1999dynamics}
\begin{equation}
\label{eqn:cantilever beam}
\frac{d^2}{dx^2}\! \left((1+x)\frac{d^2u}{dx^2}\right)=\lambda u, \quad u(0)=u'(0)=0, \quad u''(L)=u'''(L)=0.
\end{equation}
The variable coefficient $1+x$ accounts for the linear taper of the wing, while the boundary conditions on $u''$ and $u'''$ at $x=1$ express the natural requirement that the bending moment and shear force vanish at the endpoint.

To compute a few of the smoothest modes of~\cref{eqn:cantilever beam} we use the eigenfunctions $w_n$ of the cantilevered beam equation with constant coefficients, given in closed form by~\cite{han1999dynamics}
\begin{equation}
\label{eqn:constant coeff beam}
w_n(x)=\cosh\beta_n x-\cos\beta_n x+\frac{\cos\beta_n L+\cosh\beta_n L}{\sin\beta_n L+\sinh\beta_n L}(\sin\beta_n x+\sinh\beta_n x).
\end{equation}
Here $\beta_n$ is the $n$th root of $g(\beta)=\cosh(\beta L)\cos(\beta L)+1$~\cite{han1999dynamics}. We target a mode of~\cref{eqn:cantilever beam} by setting $f_0(x)=w_n(x)$. \Cref{fig:beam eqn} shows the modes that are computed using initial guesses $w_1,\dots,w_4$, corresponding to the smallest four positive roots of $g(\beta)$.

\section{Computing eigenvalues in unbounded regions}\label{sec:unbounded regions}
The stability analysis of solutions to time-dependent partial differential equations (PDEs) provides an abundant source of differential eigenvalue problems. Consider the initial boundary value problem (IBVP) with periodic boundary conditions
\begin{equation}
\label{eqn:NTPDE}
u_t=\mathcal{L}u+\mathcal{N}(u), \qquad u_t(x,0)=g(x), \qquad u(-1,t)=u(1,t).
\end{equation}
Here, $\mathcal{L}$ and $\mathcal{N}$ are linear and nonlinear ordinary differential operators (with respect to the variable $x$), respectively.
In many instances,~\cref{eqn:NTPDE} supports steady-states, traveling wave states, or other phenomena whose stability is of critical importance in the physical problem under study~\cite{laugesen2000properties,alikakos1999periodic,sanford2014stability}. When $\mathcal{L}$ is self-adjoint or normal, the stability analysis often reduces to determining whether or not the eigenvalues of $\mathcal{L}$ are contained in one half-plane~\cite{laugesen2000linear,alikakos1999periodic,sanford2014stability,trefethen2005spectra}. We now show how to modify the spectral projector in~\cref{eqn:spectral projector for diff op} to derive a practical rational filter to compute (finitely many) eigenvalues of $\mathcal{L}$ in the right half-plane.

\subsection{A rational filter for the half-plane}\label{subsec:half-plane}
Let $\mathcal{L}$ be a closed linear operator that is densely defined on a Hilbert space $\mathcal{H}$. Suppose that $\mathcal{L}$ is a normal operator with a spectrum in the left half-plane ${\rm Re}(z)<0$ except for finitely many eigenvalues $\lambda_1,\dots,\lambda_m$ (including multiplicities) such that ${\rm Re}(\lambda_i)>0$ for $1\leq i \leq m$. Denote the eigenspace associated with $\lambda_1,\dots,\lambda_m$ by $\mathcal{V}$ and consider search regions that are semi-circles of radius $R$, i.e.,
\begin{equation}
\label{eqn:half-disk}
\Omega_R=\{z\in\mathbb{C} : |z|<R, {\rm Re}(z)>0\}, \qquad R>\text{max}_{1\leq i\leq m}|\lambda_i|.
\end{equation}
To construct a computable spectral projector onto the right half-plane we consider taking $R\rightarrow\infty$. We adopt the following strategy:
\begin{itemize}[leftmargin=*,noitemsep]
\item[(i)] Introduce a $1/R$ decay into the integrand of the spectral projector~\cref{eqn:spectral projector for diff op} as $R\rightarrow\infty$, while preserving the projection onto $\mathcal{V}$.

\item[(ii)] Split the projector into an integral over the vertical part of $\partial\Omega_R$ and an integral over the circular arc of $\partial\Omega_R$. By taking $R\rightarrow\infty$, we observe that the contribution from the circular arc goes to $0$ due to the additional $1/R$ decay in the integrand.

\item[(iii)] Map the imaginary axis to the interval $[-1,1]$ and approximate the spectral projector by a quadrature rule.
\end{itemize}
\begin{figure}[!tbp]
\label{fig:RF HP}
  \centering
  \begin{minipage}[b]{0.49\textwidth}
    \centering
 \begin{tikzpicture}
   \draw[black,thick,->] (0,2)--(3,2);
   \draw[black,thick,->] (1,0)--(1,4);
   \draw(1,3.8) node[anchor=east] {$R$};
   \draw(1,0.2) node[anchor=east] {$-R$};
   \draw(3,2) node[anchor=west] {${\rm Re}$};
   \draw(1,4) node[anchor=south] {${\rm Im}$};
    \clip (1,0) rectangle (4,4);
    \draw (1,2) circle(1.8);
    \draw(2.5,1.2) node[anchor=east] {$\Omega_R$};
  \end{tikzpicture}
  \end{minipage}
  \hfill
  \begin{minipage}[b]{0.49\textwidth}
    \begin{overpic}[width=\textwidth]{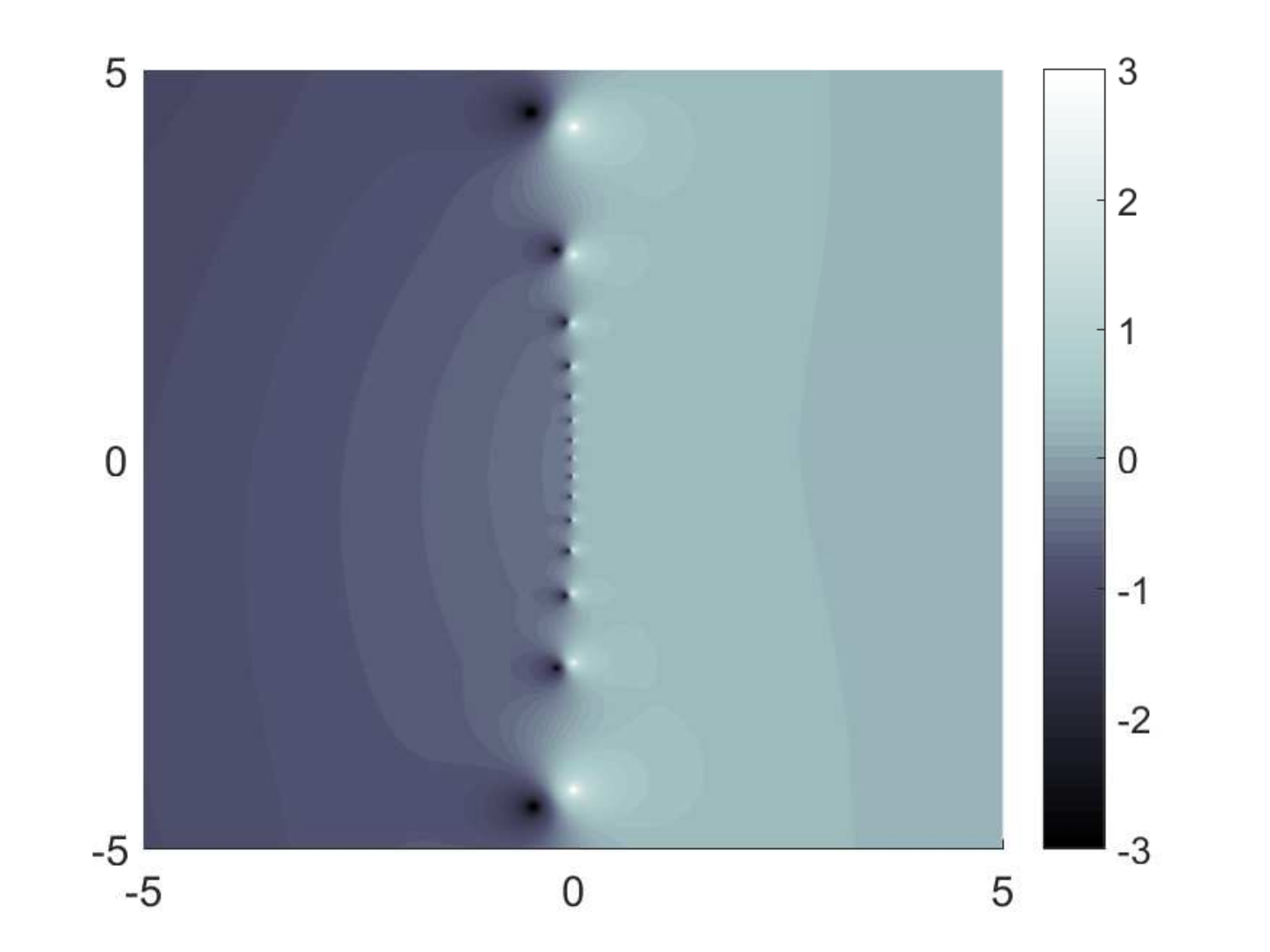}
 	\put (38,-3) {$\displaystyle{\rm Real}(z)$}
 	\put (-1,28) {\rotatebox{90} {$\displaystyle{\rm Imag}(z)$}}
	\end{overpic}
  \end{minipage}
  \caption{Left: The region $\Omega_R$ from~\cref{eqn:half-disk} used in the derivation of the rational filter over the right half-plane. Right: The constructed rational filter~\cref{eqn:HP RF} for the right half-plane with $\ell=20$.}
\end{figure}
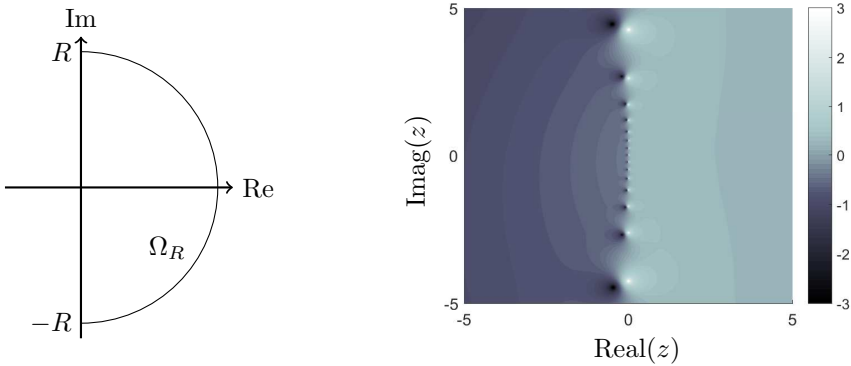

Select $a\in\mathbb{R}^+$ and consider the family of functions that are analytic in the right half-plane defined by
\begin{equation}
\label{eqn:HP projector}
\mathcal{P}_R(\lambda)=\frac{1}{2\pi i}\int_{\partial\Omega_R}(z+a)^{-1}(z-\lambda)^{-1}\, dz.
\end{equation}
By Cauchy's Integral Formula, we know that $\mathcal{P}_R(\lambda)=(\lambda+a)^{-1}$ if $\lambda\in\Omega_R$ and is zero otherwise~\cite{stein2003complex}. Taking the limit $R\rightarrow\infty$, we obtain
\begin{equation}
\label{eqn:HP proj formula}
\mathcal{P}(\lambda)=\lim_{R\rightarrow\infty}\mathcal{P}_R(\lambda)=\frac{1}{2\pi}\int_{-\infty}^\infty (iy+a)^{-1}(iy-\lambda)^{-1}\, dy. 
\end{equation}
Using functional calculus for unbounded normal operators we can extend $\mathcal{P}(\lambda)$ to an operator-valued function $\mathcal{P}(\mathcal{L})$~\cite[Theorem 13.24]{rudin1991functional}.\footnote{In~\cite[Theorem 13.24]{rudin1991functional}, $E_{x,y}$ is the spectral measure of $\mathcal{L}$~\cite[Theorem 13.33]{rudin1991functional}.} Moreover, we have that $\mathcal{P}(\mathcal{L})u=\mathcal{P}(\lambda)u$ when $\mathcal{L}u=\lambda u$~\cite[VIII.5]{reed1980methods}. Consequently, ${\rm range}(\mathcal{P}(\mathcal{L}))=\mathcal{V}$.

Now, take the change-of-variables $x=\frac{2}{\pi}\tan^{-1} y$ in~\cref{eqn:HP proj formula} to obtain
\begin{equation}
\label{eqn:HP int}
\mathcal{P}(\mathcal{L})=\frac{1}{4}\int_{-1}^1\left(i\tan\!\left(\frac{\pi x}{2}\right)+a\right)^{-1}\left(i\tan\!\left(\frac{\pi x}{2}\right)\mathcal{I}-\mathcal{L}\right)^{-1}\sec^2\!\left(\frac{\pi x}{2}\right)\, dx.
\end{equation}
Using Gauss--Legendre quadrature nodes $x_1,\dots,x_\ell$ and weights $w_1,\dots,w_\ell$ on $[-1,1]$, we can approximate $\mathcal{P}(\mathcal{L})$ by
\begin{equation}
\label{eqn:HP RF}
\mathcal{\hat P}(\mathcal{L})=\frac{1}{4}\sum_{k=1}^\ell w_k\frac{1-z_k^2}{z_k+a}(z_k\mathcal{I-L})^{-1},\qquad  z_k=i\tan\!\left(\frac{\pi x_k}{2}\right).
\end{equation}
\Cref{fig:RF HP} (right) shows the derived rational filter $\mathcal{\hat P}(\lambda)$ in the complex plane. 

\subsection{Stability of thin fluid films}\label{subsec:thin fluid films}
To demonstrate the utility of the filter in~\cref{eqn:HP RF}, we assess the stability of the steady-state solutions to a PDE governing the motion of a thin film of fluid supported below by a flat substrate. The PDE is
\begin{equation}
\label{eqn:thin fluid PDE}
u_t=\partial_x^4u+\partial_x(u\partial_x u), 
\end{equation}
where $u(x,t)$ is a positive, periodic function representing the thickness of the fluid~\cite{laugesen2000properties}. The nonlinear term models gravitational effects and substrate-fluid interactions~\cite{laugesen2000properties}.

\begin{figure}[!tbp]
\label{fig:thin-film stability}
  \centering
  \begin{minipage}[b]{0.45\textwidth}
    \begin{overpic}[width=\textwidth]{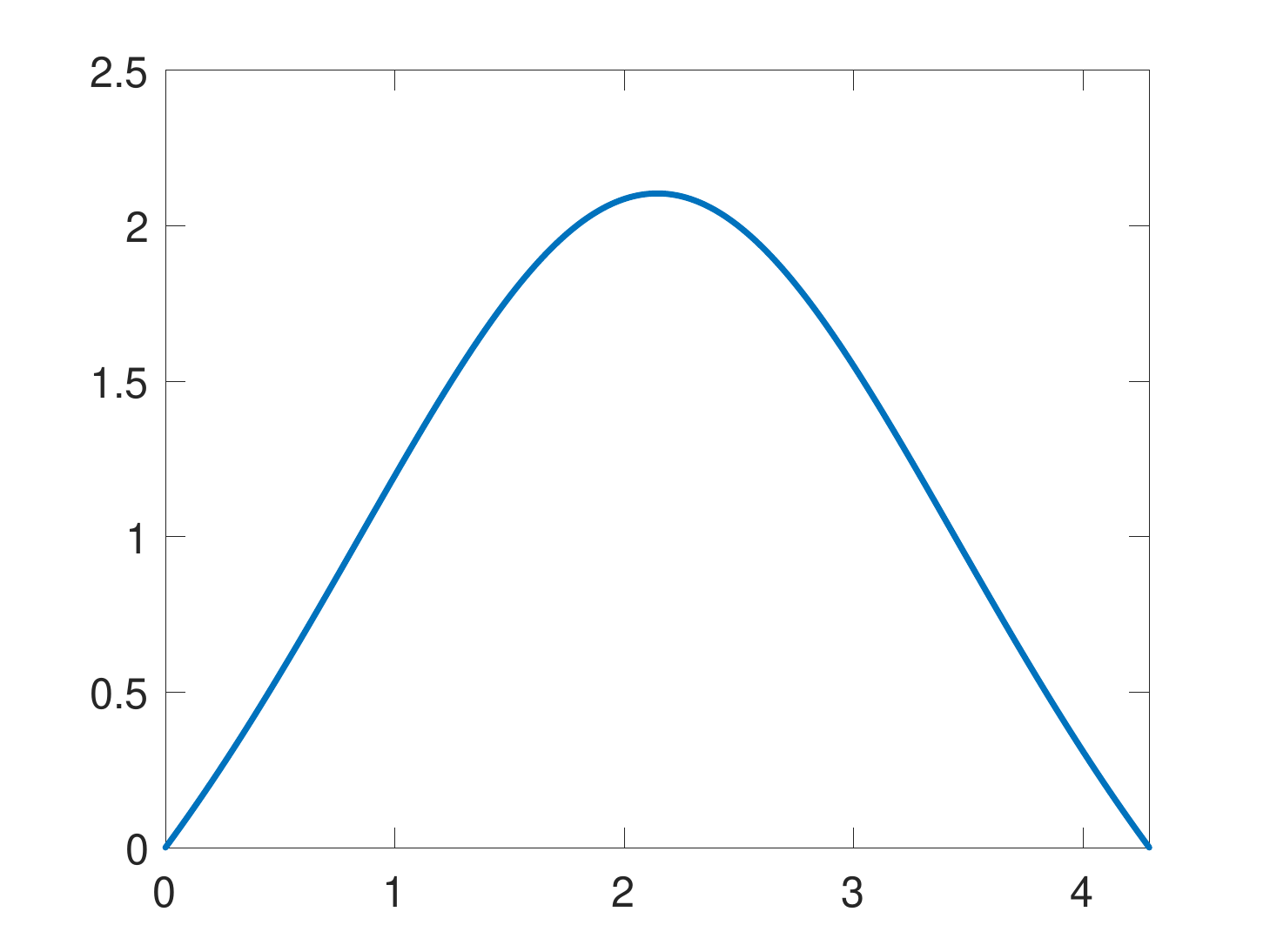}
 	\put (48,-1) {$\displaystyle x$}
 	\put(46,72) {$u_{\text{ss}}$}
	\end{overpic}
  \end{minipage}
  \hfill
  \begin{minipage}[b]{0.45\textwidth}
    \begin{overpic}[width=\textwidth]{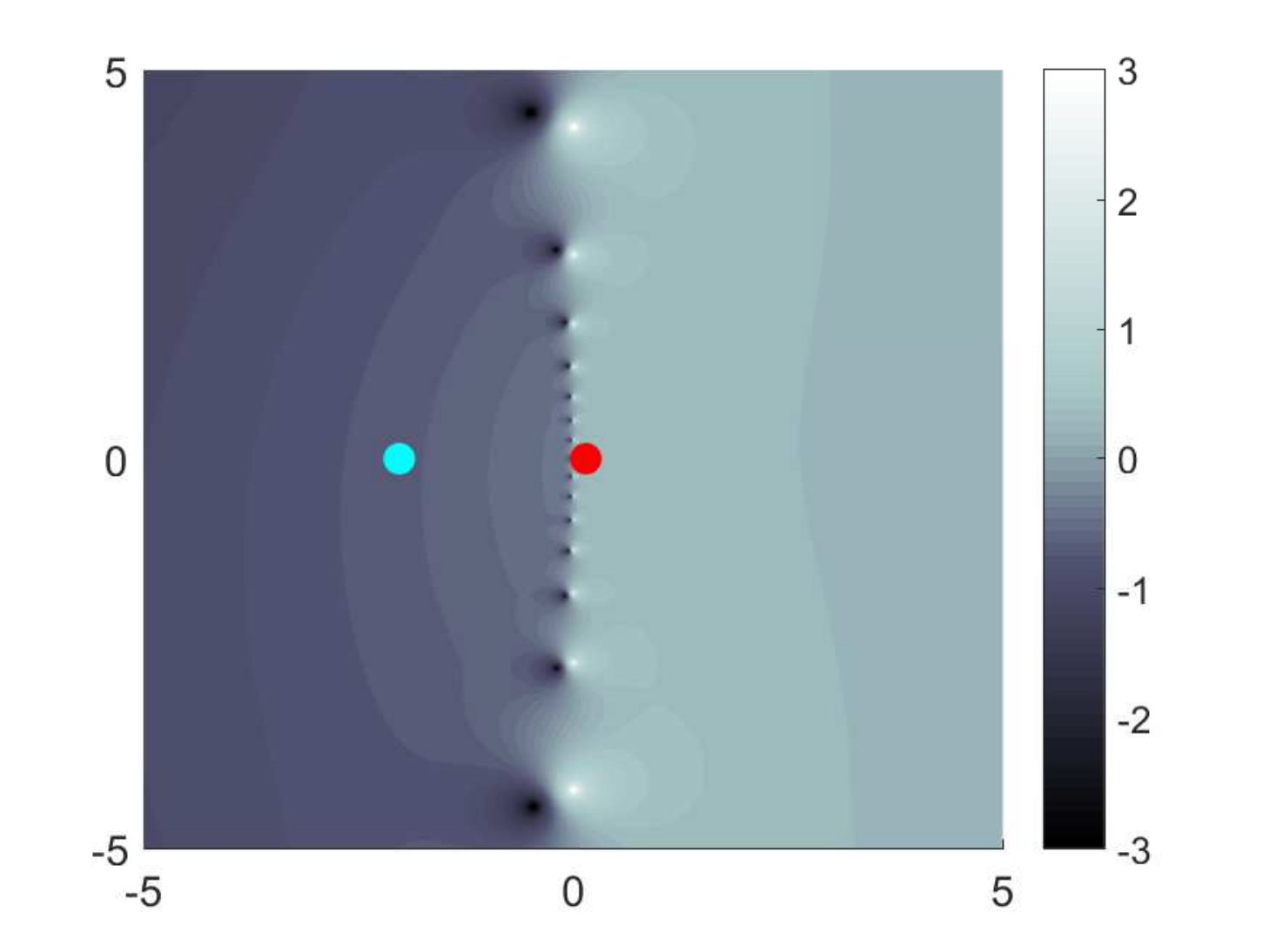}
 	\put (38,-3) {$\displaystyle{\rm Real}(z)$}
 	\put (-1,28) {\rotatebox{90} {$\displaystyle{\rm Imag}(z)$}}
	\end{overpic}
  \end{minipage}
  \caption{Left: A droplet, $u_{\text{ss}}$, which is a steady-state solution to~\cref{eqn:thin fluid PDE}, computed from the IVP~\cref{eqn:steady-state}. Right: Two rightmost eigenvalues (blue and red dots) of~\cref{eqn:stability of touchdown states} together with a log-scale colormap of the rational filter in~\cref{eqn:HP RF} with $\ell=20$, which is used in place of~\cref{eqn:approximate spectral projector}. The eigenvalue with a positive real part (red dot) indicates that this steady-state droplet is unstable.}
\end{figure}

A droplet steady-state $u_{\text{ss}}(x)$ of~\cref{eqn:thin fluid PDE}, rescaled so that it is supported on $[0, l]$ with contact angle $\pi/4$, is stable if all the eigenvalues of a fourth-order differential operator are in the left half-plane. The associated differential eigenproblem is~\cite{laugesen2000linear}
\begin{equation}
\label{eqn:stability of touchdown states}
\frac{d^4u}{dx^4}+\frac{d}{dx}\!\left(u_{\text{ss}}\frac{du}{dx}\right)=\lambda u, \qquad u(0)=u(l)=0, \quad u''(0)=u''(l)=0.
\end{equation}
We compute the steady-state $u_{\text{ss}}(x)$ by solving the second order nonlinear ODE~\cite{laugesen2000properties}
\begin{equation}
\label{eqn:steady-state}
\frac{du_{\text{ss}}}{dx}+\frac{1}{2}u_{\text{ss}}^2-\delta=0, \qquad u_{\text{ss}}(0)=0,\quad u_{\text{ss}}'(0)=1.
\end{equation}
Here, $\delta$ is a dimensionless quantity relating the rescaled problem to the original contact angle~\cite{laugesen2000properties}. The length $l$ of the droplet's base and $\delta$ may be calculated analytically~\cite{laugesen2000properties}.

In~\cref{fig:thin-film stability}, we show an approximation to the rescaled steady-state $u_{\text{ss}}$ along with the right-most eigenvalues of~\cref{eqn:stability of touchdown states}. Using the rational filter in~\cref{eqn:HP RF} with $\ell=20$ (the degree of the quadrature rule defining the filter) to perform the approximate spectral projection in Algorithm~\ref{alg:contFEAST}, we are able to identify an eigenvalue of~\cref{eqn:stability of touchdown states} in the right half-plane, which indicates that the droplet (see~\cref{fig:thin-film stability} (left)) is unstable.

Techniques for selecting the dimension $m$ of the subspace $\mathcal{V}$~\cite{SubspaceIter,kestyn2016feast} are important in stability analysis as one is trying to determine the number of eigenvalues in the right half-plane. To select $m$, we monitor the singular values of the matrix $\hat V^*\hat V$ after each iteration and adjust the number of basis functions by removing columns of $\hat V$ associated with singular values that are close to machine precision (relative to the largest singular value)~\cite{SubspaceIter}. This procedure usually allows us to capture the dominant eigenspace of the filtered operator $\mathcal{\hat P}(\mathcal{L})$ that includes the target eigenspace as well as any eigenvalues clustered near the imaginary axis. We then determine whether there are any eigenvalues in the right half-plane by sorting through the computed eigenvalues. However, this strategy may break down, for instance, if there is an eigenvalue close to a quadrature node. Additionally, the sharp decay of the filter~\cref{eqn:HP RF} across the imaginary axis is softened as $\lvert{\rm Im}(z)\rvert\rightarrow\infty$, which can lead to difficulties when there are clusters of eigenvalues near the imaginary axis with large imaginary part. In this case, one may need to take a large number of basis functions to accurately resolve the dominant eigenvalues of $\mathcal{\hat P}(\mathcal{L})$.

\section*{Conclusions}\label{sec:conclusions}
An operator analogue of the FEAST matrix eigensolver is derived to solve differential eigenvalue problems without discretizing the operator. This approach leads to an algorithm that can exploit spectrally accurate techinques for computing with functions while preserving the structure of $\mathcal{L}$. The result is an efficient, automated, and accurate eigensolver for normal and self-adjoint differential operators. This eigensolver is adept in the high-frequency regime and may provide a new direction towards robust high-frequency eigenvalue computations. 

The implementation described in~\cref{sec:practical algorithm} extends to higher dimensions in a straightforward way for simple geometries where spectral methods apply~\cite{townsend2015automatic}. In the case of more complicated geometries, one may still benefit from the advantages of the ``solve-then-discretize" framework outlined in Algorithm~\ref{alg:contFEAST} provided that one is able to accurately compute inner products and solutions to the shifted linear differential equations. Although we have focused on the strong form of the eigenvalue problem in~\cref{eqn:differential eigenvalue problem}, it may be necessary to work with the weak form. In this case, one can still follow Algorithm~\ref{alg:contFEAST} provided that the shifted linear systems are solved in weak form and the Rayleigh--Ritz projection is performed with the associated sesquilinear form.

\section*{Acknowledgements}
We thank Anthony Austin for providing detailed feedback on the results in~\cref{sec:contFEAST} and on our rational filter in~\cref{sec:unbounded regions}. We thank David Bindel for helpful discussions about spectral discretizations of differential eigenvalue problems. Finally, we thank Marcus Webb, Sheehan Olver, Yuji Nakatsukasa, Heather Wilber, and Matthew Colbrook for providing comments on a draft.

\bibliography{draft}
\bibliographystyle{plain}
\end{document}